\documentclass[12pt,oneside,english]{amsart}
\usepackage[T1]{fontenc}
\usepackage[utf8]{inputenc}
\usepackage{geometry}
\usepackage{amstext}
\usepackage{amsthm}
\usepackage{amssymb}
\usepackage{graphicx}
\usepackage{setspace}
\usepackage{hyperref}
\makeatletter

\newcommand{\noun}[1]{\textsc{#1}}

\numberwithin{equation}{section}
\numberwithin{figure}{section}
\theoremstyle{plain}
\newtheorem{thm}{\protect\theoremname}
\theoremstyle{definition}
\newtheorem{defn}[thm]{\protect\definitionname}
\newenvironment{lyxlist}[1]
	{\begin{list}{}
		{\settowidth{\labelwidth}{#1}
		 \setlength{\leftmargin}{\labelwidth}
		 \addtolength{\leftmargin}{\labelsep}
		 }}
	{\end{list}}
\theoremstyle{remark}
\newtheorem{rem}[thm]{\protect\remarkname}
\theoremstyle{plain}
\newtheorem{lem}[thm]{\protect\lemmaname}
\theoremstyle{plain}
\newtheorem{prop}[thm]{\protect\propositionname}
\theoremstyle{plain}
\newtheorem{cor}[thm]{\protect\corollaryname}

\makeatother

\usepackage{babel}
\addto\captionsenglish{\renewcommand{\corollaryname}{Corollary}}
\addto\captionsenglish{\renewcommand{\definitionname}{Definition}}
\addto\captionsenglish{\renewcommand{\lemmaname}{Lemma}}
\addto\captionsenglish{\renewcommand{\propositionname}{Proposition}}
\addto\captionsenglish{\renewcommand{\remarkname}{Remark}}
\addto\captionsenglish{\renewcommand{\theoremname}{Theorem}}
\addto\captionsngerman{\renewcommand{\corollaryname}{Korollar}}
\addto\captionsngerman{\renewcommand{\definitionname}{Definition}}
\addto\captionsngerman{\renewcommand{\lemmaname}{Lemma}}
\addto\captionsngerman{\renewcommand{\propositionname}{Satz}}
\addto\captionsngerman{\renewcommand{\remarkname}{Bemerkung}}
\addto\captionsngerman{\renewcommand{\theoremname}{Theorem}}
\providecommand{\corollaryname}{Corollary}
\providecommand{\definitionname}{Definition}
\providecommand{\lemmaname}{Lemma}
\providecommand{\propositionname}{Proposition}
\providecommand{\remarkname}{Remark}
\providecommand{\theoremname}{Theorem}

\begin{document}
\begin{titlepage}
\addtocounter{page}{1}
\begin{center}
\includegraphics[scale=0.4]{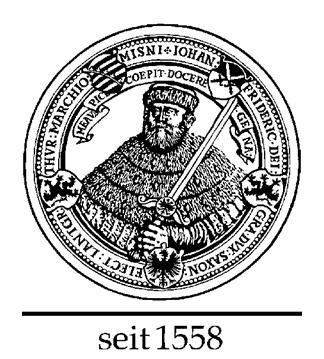}\\
\par\end{center}
\title{ENTROPY-, APPROXIMATION- AND KOLMOGOROV NUMBERS ON QUASI-BANACH SPACES}
\makeatletter
\begin{center}
  {\large\bfseries \@title\par}
\end{center}
\makeatother
\begin{doublespace}
\begin{center}
Bachelor Thesis
\par\end{center}

\begin{center}
to achieve the Academic Degree 
\par\end{center}

\begin{center}
Bachelor of Science (B. Sc.)
\par\end{center}

\begin{center}
in Mathematics
\par\end{center}

\begin{center}
FRIEDRICH-SCHILLER-UNIVERSIT\"{A}T JENA
\par\end{center}

\begin{center}
Fakult\"{a}t f\"{u}r Mathematik und Informatik 
\par\end{center}

\begin{center}
submitted by Marcus Gerhold
\par\end{center}

\begin{center}
born 30/09/1989 in Hohenm\"{o}lsen
\par\end{center}

\begin{center}
Adviser: Prof. Dr. Dorothee D. Haroske 
\par\end{center}

\begin{center}
Jena, 25/08/2011
\par\end{center}
\end{doublespace}
\end{titlepage}
\clearpage
\thispagestyle{empty}
\noindent\textsc{Abstract.}
\thispagestyle{empty}
In this bachelor's thesis we introduce three quantities for linear
and bounded operators on quasi-Banach spaces which are entropy numbers,
approximation numbers and Kolmogorov numbers. At first we establish
the three quantities with some basic properties and try to modify
known content from the Banach space case. We compare each one of them,
with the corresponding other two and give estimates concerning the
mean values and limits. As an example, we analyze the identity operator
between finite dimensional $\ell_{p}$ spaces $\mbox{id : }\left(\ell_{p}^{n}\rightarrow\ell_{q}^{n}\right)$
for $0<p,q\leq\infty$ and give sharp estimates for entropy numbers.
Furthermore we add some known estimates for approximation numbers
and Kolmogorov numbers. At last we examine some renowned connections
of these quantities to spectral theory on infinite dimensional Hilbert
spaces, which are the inequality of Carl\noun{ }and the inequality
of\noun{ }Weyl\noun{. }

\thispagestyle{empty}
\selectlanguage{ngerman}%
\newpage{}

\selectlanguage{english}%
\tableofcontents{}

\newpage{}
An expression of many thanks goes to my adviser Prof. Dr. Dorothee
Haroske for the constant support and always taking a lot of time for
my concerns.

\thispagestyle{empty}

\selectlanguage{ngerman}%
\newpage{}
\selectlanguage{english}%

\part{Preparations}

\section{Norms}

As the title suggests, this bachelor's thesis is about entropy-, approximation-
and Kolmogorov numbers of linear and bounded operators acting between
quasi-Banach spaces. One of the ideas of these quantities is to give
a measure of how compact a compact operator is. To establish these
numbers, we will first need to clarify notations and basic definitions
concerning norms, Banach spaces, $\ell_{p}$ spaces, operators and
quotient maps, which is the main aim of this part. Therefore we start
with the definitions of norms, quasi-norms and $\varrho$- norms.
\begin{defn}
\label{def:Definition Norm}Let $\mathbb{X}$ be a linear space over
a field $\mathbb{K}$. A mapping $\Vert\cdot\Vert\,:\,\mathbb{X}\rightarrow[0,\infty)$
is called norm if it satisfies the following three conditions
\end{defn}

\begin{lyxlist}{00.00.0000}
\item [{(i)}] $\Vert x\Vert=0\Longleftrightarrow x=0$
\item [{(ii)}] $\Vert\lambda x\Vert=\vert\lambda\vert\,\Vert x\Vert$ for
$\lambda\in\mathbb{K}$, $x\in\mathbb{X}$
\item [{(iii)}] $\Vert x+y\Vert\leq\Vert x\Vert+\Vert y\Vert$ for $x,y\in\mathbb{X}$.
\end{lyxlist}
If condition (iii) is substituted by (iii') $\Vert x+y\Vert\leq C\left(\Vert x\Vert+\Vert y\Vert\right)$
with a constant $C\geq1$, which is independent of $x$ and $y$,
we call $\Vert\cdot\Vert$ quasi-norm. If condition (iii) is substituted
by (iii'') $\Vert x+y\Vert^{\varrho}\leq\Vert x\Vert^{\varrho}+\Vert y\Vert^{\varrho}$
for $\varrho\in\left(0,1\right]$, we call $\Vert\cdot\Vert$ $\varrho$
- norm. 

If we talk of Banach spaces as complete normed vector spaces, it is
only natural that the terms quasi-Banach space and $\varrho$ - Banach
space arise, if a vector space is complete with respect to the corresponding
quasi- or $\varrho$ - norm.

As we will show in the following theorem, which can be found in \cite[Ch. 2, Thm. 1.1.]{Constructive Approximation},
it does not matter whether we are dealing with quasi-Banach spaces
or $\varrho$ - Banach spaces, since both are equivalent.
\begin{rem}
As the notation of the closed and open unit balls differ, we will
use

\[
B_{\mathbb{X}}=\left\{ x\in\mathbb{X}:\,\Vert x\Vert_{\mathbb{X}}<1\right\} \,\,\,\mbox{and}\,\,\,\overline{B}_{\mathbb{X}}=\left\{ x\in\mathbb{X}:\,\Vert x\Vert_{\mathbb{X}}\leq1\right\} 
\]

as the open and correspondingly the closed unit ball for a given normed
space.
\end{rem}

\begin{thm}
\label{thm:Normequivalence}Let $\mathbb{X}$ be a linear space. For
each quasi-norm $\Vert\cdot\Vert$ with constant $C\geq1$ exists
an equivalent $\varrho$ - norm $\Vert\cdot\Vert_{0}$ with $\varrho\in\left(0,1\right]$.
\end{thm}

\begin{proof}
With $C$ as constant of the quasi-norm $\Vert\cdot\Vert$ we define
$C_{0}:=2C$. With $C_{0}\geq2$, and $f,g\in\mathbb{X}$ we get the
result $\Vert f+g\Vert\leq C\left(\Vert f\Vert+\Vert g\Vert\right)\leq2C\max\left\{ \Vert f\Vert,\Vert g\Vert\right\} =C_{0}\max\left\{ \Vert f\Vert,\Vert g\Vert\right\} $.
By induction we derive that for $f_{1},\ldots,f_{m}\in\mathbb{X}$,
$m\in\mathbb{N}$

\begin{equation}
\Vert f_{1}+\ldots+f_{m}\Vert\leq\max_{1\leq j\leq m}\left(C_{0}^{j}\Vert f_{j}\Vert\right).\label{eq:Theorem1.a}
\end{equation}

Since $C$ is given, we define $\varrho$ by $C_{0}^{\varrho}=2$
and the mapping $\Vert\cdot\Vert_{0}\,:\,\mathbb{X}\,\rightarrow[0,\infty)$
by

\begin{equation}
\Vert f\Vert_{0}:=\inf_{f=f_{1}+\ldots+f_{m}}\left\{ \Vert f_{1}\Vert^{\varrho}+\ldots+\Vert f_{m}\Vert^{\varrho}\right\} ^{1/\varrho},\label{eq:Theorem1.b}
\end{equation}

where the infimum is taken over all decompositions of $f$. At first
we observe, that $C_{0}^{\varrho}=2\Longleftrightarrow\varrho=\frac{\ln2}{\ln C_{0}}$,
hence $\varrho\in(0,1]$. We now want to show that $\Vert\cdot\Vert_{0}$
is our desired $\varrho$ - norm. Therefore we need to investigate
the three properties of Definition \ref{def:Definition Norm} of the
mapping. Properties (i) and (ii) of $\varrho$ - norms can be derived
immediately. Property (iii'') is given through

\begin{eqnarray*}
\Vert f+g\Vert_{0}^{\varrho} & = & \inf_{f+g=h_{1}+\ldots+h_{m}}\left\{ \Vert h_{1}\Vert^{\varrho}|+\ldots+\Vert h_{m}\Vert^{\varrho}\right\} \\
 & \leq & \inf_{f=f_{1}+\ldots+f_{m}}\left\{ \Vert f_{1}\Vert^{\varrho}+\ldots+\Vert f_{m}\Vert^{\varrho}\right\} \\
 &  & +\inf_{g=g_{1}+\ldots+g_{m}}\left\{ \Vert g_{1}\Vert^{\varrho}+\ldots+\Vert g_{m}\Vert^{\varrho}\right\} \\
 & = & \Vert f\Vert_{0}^{\varrho}+\Vert g\Vert_{0}^{\varrho}.
\end{eqnarray*}

Hence $\Vert\cdot\Vert_{0}$ is a $\varrho$ - norm. Now we need to
show that $\Vert\cdot\Vert$ and $\Vert\cdot\Vert_{0}$ are equivalent
by showing, that there exist constants $a,\,A>0$ in a way, that $a\Vert\cdot\Vert_{0}\leq\Vert\cdot\Vert\leq A\Vert\cdot\Vert_{0}$.
The first inequality can easily be shown. Since the infimum in $\Vert\cdot\Vert_{0}$
is taken over all decompositions, it is also taken over the trivial
decomposition $f=f$ , which yields $\Vert f\Vert_{0}=\inf_{f=f_{1}+\ldots+f_{m}}\left\{ \Vert f_{1}\Vert^{\varrho}+\ldots+\Vert f_{m}\Vert^{\varrho}\right\} ^{1/\varrho}\leq\inf_{f=f_{1}}\left\{ \Vert f_{1}\Vert^{\varrho}\right\} ^{1/\varrho}=\Vert f\Vert$.
Thus $a=1$. 

For the other inequality we define

\[
N(f)=\begin{cases}
0 & \mbox{ if }f=0\\
C_{0}^{k} & \mbox{ if }C_{0}^{k-1}<\Vert f\Vert\leq C_{0}^{k}.
\end{cases}
\]

Clearly we get 

\begin{equation}
C_{0}^{-1}N(f)\leq\Vert f\Vert\leq N(f).\label{eq:Theorem1.d}
\end{equation}

At first we show by induction 

\begin{equation}
\Vert f_{1}+\ldots+f_{m}\Vert\leq C_{0}\left(N(f_{1})^{\varrho}+\ldots+N(f_{m})^{\varrho}\right)^{1/\varrho}.\label{eq:Theorem1.c}
\end{equation}

The case $m=1$ follows immediately. We suppose that \eqref{eq:Theorem1.c}
has been established for $m=n-1$. Now for given $f_{1},\ldots,f_{n}\in\mathbb{X}$
we can assume without loss of generality, that $\Vert f_{1}\Vert\geq\ldots\geq\Vert f_{n}\Vert$
(otherwise we renumber all $f_{i}$). If all $N(f_{i})$, $i=1,\ldots,n$
are distinct, we have

\[
C_{0}^{j}||f_{j}||\overset{\eqref{eq:Theorem1.d}}{\leq}C_{0}^{j}N(f_{j})\leq C_{0}N(f_{1})\leq C_{0}\left(N(f_{1})^{\varrho}+\ldots+N(f_{n})^{\varrho}\right)^{1/\varrho}.
\]

\eqref{eq:Theorem1.c} follows from \eqref{eq:Theorem1.a}. 

Let us now consider the case, where for certain $j\in\left\{ 1,\ldots,n-1\right\} $
$N(f_{j})=N(f_{j+1})=C_{0}^{l}$ , $l\in\mathbb{Z}$. Using \eqref{eq:Theorem1.a},
we have $\Vert f_{j}+f_{j+1}\Vert\leq C_{0}\max\left\{ \Vert f_{j}\Vert,\Vert f_{j+1}\Vert\right\} =C_{0}^{l+1}$
and furthermore $N(f_{j}+f_{j+1})^{\varrho}\leq C_{0}^{\varrho(l+1)}=2^{(l+1)}=2^{l}+2^{l}=N(f_{j})^{\varrho}+N(f_{j+1})^{\varrho}$.
Combining this result with our induction hypothesis, yields

\begin{eqnarray*}
\Vert f_{1}+\ldots+f_{m}\Vert & \overset{\mbox{i.h.}}{\leq} & C_{0}\left(N(f_{1})^{\varrho}+\ldots+N(f_{j}+f_{j+1})^{\varrho}+\ldots+N(f_{m})^{\varrho}\right)^{1/\varrho}\\
 & \leq & C_{0}\left(N(f_{1})^{\varrho}+\ldots+N(f_{j})^{\varrho}+N(f_{j+1})^{\varrho}+\ldots+N(f_{m})^{\varrho}\right)^{1/\varrho}.
\end{eqnarray*}

Thus we have proved \eqref{eq:Theorem1.c} from which the wanted inequality
can be derived from

\begin{eqnarray*}
\Vert f_{1}+\ldots+f_{m}\Vert & \leq & C_{0}^{2}\left(\left(\frac{N(f_{1})}{C_{0}}\right)^{\varrho}+\ldots+\left(\frac{N(f_{m})}{C_{0}}\right)^{\varrho}\right)^{1/\varrho}\\
 & \overset{\eqref{eq:Theorem1.d}}{\leq} & C_{0}^{2}\left(\Vert f_{1}\Vert^{\varrho}+\ldots+\Vert f_{m}\Vert^{\varrho}\right)^{1/\varrho}.
\end{eqnarray*}

Now we can take the infimum over all decompositions of $f$ (as it
is done in the definition of $\Vert\cdot\Vert_{0}$) and get $A=C_{0}^{2}$,
and hence the equivalence of the norms $\Vert\cdot\Vert$ and $\Vert\cdot\Vert_{0}$. 
\end{proof}
This theorem will already be useful in the next statement, which is
well known and part of many lectures. Its proof, in the case for Banach
spaces can be found in \cite[Ch. 3, Thm. 1.1.]{Constructive Approximation},
but is done here for quasi-Banach spaces.
\begin{thm}
\label{thm:bestapproximation}Let $\left[\mathbb{X},\Vert\cdot\Vert\right]$
be a quasi-normed, linear vector space with constant $C_{\mathbb{X}}$
and $U\subset\mathbb{X}$ a linear subspace with $\dim U<n<\infty$.
Then for every element $f\mathbb{\in X}$, there exists an element
$g\in U$ with

\[
\Vert f-g\Vert=\inf_{h\in U}\Vert f-h\Vert.
\]
\end{thm}

\begin{proof}
As we have shown in the preceding theorem, we find an equivalent $\varrho$
- Norm for every quasi-norm. So let $\Vert\cdot\Vert_{0}$ be an equivalent
$\varrho$ - norm for the quasi-norm $\Vert\cdot\Vert$. By definition
of the infimum, there exists a sequence $\left(h_{k}\right)_{k\in\mathbb{N}}\in U$
for which $\Vert f-h_{n}\Vert_{0}\overset{n\rightarrow\infty}{\longrightarrow}\inf_{h\in U}\Vert f-h\Vert_{0}$.
We also have $\Vert h_{n}\Vert_{0}^{\varrho}\leq\Vert f\Vert_{0}^{\varrho}+\Vert f-h_{n}\Vert_{0}^{\varrho}$.
This means, that $\left(h_{n}\right)_{n\in\mathbb{N}}$ is a bounded
sequence on a finite dimensional subspace, hence $\left(h_{n}\right)_{n\in\mathbb{N}}$
is relatively compact . Therefore we can find a subsequence $\left(h_{n_{j}}\right)_{j\in\mathbb{N}}\subset\left(h_{n}\right)_{n\in\mathbb{N}}$
and an element $g\in\mathbb{X}$ with $\Vert h_{n_{j}}-g\Vert_{0}\overset{j\rightarrow\infty}{\longrightarrow}0$.
Furthermore we get
\[
\Vert f-g\Vert_{0}^{\varrho}\leq\Vert f-h_{n_{j}}\Vert_{0}^{\varrho}+\Vert h_{n_{j}}-g\Vert_{0}^{\varrho}\leq\Vert f-g\Vert_{0}^{\varrho}+\Vert g-h_{n_{j}}\Vert+\Vert h_{n_{j}}-g\Vert_{0}^{\varrho}\overset{j\rightarrow\infty}{\longrightarrow}\Vert f-g\Vert_{0}^{\varrho}
\]

and by taking the $\varrho$ - th root $\Vert f-h_{n_{j}}\Vert_{0}\overset{j\rightarrow\infty}{\longrightarrow}\Vert f-g\Vert_{0}$.
On the other hand $\Vert f-h_{n}\Vert_{0}\overset{n\rightarrow\infty}{\longrightarrow}\inf_{h\in U}\Vert f-h\Vert_{0}$
and therefore $\Vert f-g\Vert_{0}=\inf_{h\in U}\Vert f-h\Vert_{0}$.
In particular we gain $g\in U$. With the established equivalence,
this result is also valid for the quasi norm $\Vert\cdot\Vert$.
\end{proof}

\section{Sequence Spaces $\ell_{p}^{n}$ and $\ell_{p}$}
\begin{defn}
For given $0<p<\infty$, $n\in\mathbb{N}$ and a field $\mathbb{K}$
(which is $\mathbb{R}$ or $\mathbb{C}$) we define
\end{defn}

\begin{lyxlist}{00.00.0000}
\item [{(i)}] $\ell_{p}^{n}:=\left\{ a\in\mathbb{K}^{n}:\,\sum_{j=1}^{n}\vert a_{j}\vert^{p}<\infty\right\} $
\\
and $\Vert\cdot\Vert_{p}:=\left(\sum_{j=1}^{n}\vert a_{j}\vert^{p}\right)^{1/p}$ 
\item [{(ii)}] $\ell_{\infty}^{n}:=\left\{ a\in\mathbb{K}^{n}:\,\sup_{1\leq j\leq n}a_{j}<\infty\right\} $
\\
and $\Vert\cdot\Vert_{\infty}:=\sup_{1\leq j\leq n}\vert a_{j}\vert$ 
\item [{(iii)}] $\ell_{p}\left(\mathbb{N}\right):=\left\{ a=\left(a_{j}\right)_{j\in\mathbb{N}}\subset\mathbb{K}:\,\sum_{j=1}^{\infty}\vert a_{j}\vert^{p}<\infty\right\} $
\\
and $\Vert\cdot\Vert_{p}:=\left(\sum_{j=1}^{\infty}\vert a_{j}\vert^{p}\right)^{1/p}$ 
\item [{(iv)}] $\ell_{\infty}:=\left\{ a=\left(a_{j}\right)_{j\in\mathbb{N}}\subset\mathbb{K}:\,\sup_{j\in\mathbb{N}}\vert a_{j}\vert<\infty\right\} $
\\
and $\Vert\cdot\Vert_{\infty}:=\sup_{j\in\mathbb{N}}\vert a_{j}\vert$ 
\end{lyxlist}
\begin{rem}
Since the sequence spaces are well known, we shall use the following
two statements without proof in this bachelor's thesis. They can be
found in \cite[Sect. 1.2.]{H=0000F6here Analysis} 
\end{rem}

\begin{enumerate}
\item For $0<p<1$ $\left[\ell_{p}\left(\mathbb{N}\right);\Vert\cdot\Vert_{p}\right]$
are quasi-Banach spaces.
\item For $1\leq p\leq\infty$ $\left[l_{p}\left(\mathbb{N}\right);\,\Vert\cdot\Vert_{p}\right]$
are Banach spaces.
\end{enumerate}

\section{Linear and Compact Operators}
\begin{defn}
Let $\left[\mathbb{X},\left\Vert \cdot\right\Vert _{\mathbb{X}}\right]$
and $\left[\mathbb{Y},\left\Vert \cdot\right\Vert _{\mathbb{Y}}\right]$
be quasi-normed spaces.
\end{defn}

\begin{lyxlist}{00.00.0000}
\item [{(i)}] A linear mapping $A\,:\,\mathbb{X}\longrightarrow\mathbb{Y}$
is called bounded operator, \\
if $\forall x\in\mathbb{X}\,\,\exists c>0:\,\Vert Ax\Vert_{\mathbb{Y}}\leq c\Vert x\Vert_{\mathbb{X}}$.
\item [{(ii)}] $\mathcal{L}\left(\mathbb{X},\mathbb{Y}\right):=\left\{ A\,:\,\mathbb{X}\longrightarrow\mathbb{Y},\,A\mbox{ linear and bounded}\right\} $
(If $\mathbb{X}=\mathbb{Y}$ we will write $\mathcal{L}\left(\mathbb{X}\right):=\mathcal{L}\left(\mathbb{X},\mathbb{X}\right).$)
\item [{(iii)}] If $T\in\mathcal{L}\left(\mathbb{X},\mathbb{Y}\right)$
we define $\mathcal{R}\left(T\right)=\left\{ y\in\mathbb{Y}:\,\exists x\in\mathbb{X}:\,Tx=y\right\} $
as range of the operator $T$.
\item [{(iv)}] An operator $A\in\mathcal{L}\left(\mathbb{X},\mathbb{Y}\right)$
is called compact, if the range of every bounded set in $\mathbb{X}$
is relatively compact in $\mathbb{Y}.$
\item [{(v)}] $\mathcal{K}\left(\mathbb{X},\mathbb{Y}\right):=\left\{ A\in\mathcal{L}\left(\mathbb{X},\mathbb{Y}\right),\,A\mbox{ compact}\right\} $.
\item [{(vi)}] $\mbox{rank }T:=\dim\mathcal{R}\left(T\right)$
\item [{(vii)}] $A$ is called finite rank operator if $\mathcal{R}\left(A\right)\subseteq\mathbb{Y}$
and $\mbox{rank }A<\infty$.
\end{lyxlist}
\begin{rem}
\label{rem:Operator property} Again, these definitions as well as
many conclusions of them are well known and we will take them for
granted. For completeness we shall list a few, which will be used
in this bachelor's thesis. Their proofs however can be found in \cite[Folg. 1.18.; Bem. in 2.1.3.; Folg. 2.12.]{HaroskeHA}.
\end{rem}

\begin{enumerate}
\item If $\left[\mathbb{X},d\right]$ is an arbitrary complete metric space
and $A\subset\mathbb{X}$, then $A$ is relatively compact if and
only if there exists a finite $\varepsilon$ - net for $A$ for every
$\varepsilon>0$.
\item $K\in\mathcal{K}\left(\mathbb{X},\mathbb{Y}\right)$ if and only if
$K\left(\overline{B}_{\mathbb{X}}\right)$ is relatively compact in
$\mathbb{Y}$.
\item Let $\left[\mathbb{X},\Vert\cdot\Vert_{\mathbb{X}}\right]$ be a quasi-normed
space, $\left[\mathbb{Y},\Vert\cdot\Vert_{\mathbb{Y}}\right]$ a quasi-Banach
space and $A\in\mathcal{L}\left(\mathbb{X},\mathbb{Y}\right)$. If
there exists a sequence of operators $\left(A_{n}\right)_{n}\subset\mathcal{L}\left(\mathbb{X},\mathbb{Y}\right)$
and $\mbox{rank }A_{k}=n_{k}<\infty\,\forall k\in\mathbb{N}$ for
which $\Vert A-A_{n}\Vert\overset{n\rightarrow\infty}{\longrightarrow}0$.
Then $A\in\mathcal{K}\left(\mathbb{X},\mathbb{Y}\right)$. (As mentioned
above, the proof can be found in \cite[Folg. 2.12.]{HaroskeHA} for
Banach spaces. It is certified quickly that the fact stays correct
for quasi-Banach spaces.)
\end{enumerate}

\section{Quotient Spaces and the Quotient Map}
\begin{defn}
\label{def:Quotientendefinition}Let $\left[\mathbb{X},\Vert\cdot\Vert_{\mathbb{X}}\right]$
a normed vector space over a field $\mathbb{K}$ and $U\subset\mathbb{X}$
linear subspace of $\mathbb{X}$ .
\end{defn}

\begin{lyxlist}{00.00.0000}
\item [{(i)}] We call $\mathbb{X}/U:=\left\{ x+U;\,x\in\mathbb{X}\right\} =\left\{ \left[x\right]_{U}:\,x\in\mathbb{X}\right\} $
the quotient space of $\mathbb{X}$ with respect to $U$.
\item [{(ii)}] The mapping $Q_{U}^{\mathbb{X}}:\,\mathbb{X}\longrightarrow\mathbb{X}/U$
is called canonical quotient map of $\mathbb{X}$ with respect to
$U$.
\item [{(iii)}] We define addition as $\left[x\right]_{U}+\left[y\right]_{U}=x+y+U=\left[x+y\right]_{U}$
and multiplication with a scalar $\lambda\in\mathbb{K}$ as $\lambda\left[x\right]_{U}=\lambda x+U=\left[\lambda x\right]_{U}$.
\item [{(iv)}] For $x\in\mathbb{X}$ we define $\Vert\left[x\right]_{U}\Vert_{\mathbb{X}/U}:=\inf_{y\in U}\Vert x+y\Vert_{\mathbb{X}}$
\end{lyxlist}
\begin{rem}
Since we only need one result, following from these definitions, which
is that $\left[\mathbb{X}/U,\,\Vert\cdot\Vert_{\mathbb{X}/U}\right]$
is a normed vector space, we won't prove it in this bachelor's thesis.
The proof can be found in \cite[Satz 3.13.]{HaroskeHA}. 
\end{rem}

The following lemma is taken from \cite[p. 49]{Carl=000026Stephani}
and is slightly modified here.
\begin{lem}
\label{lem:quotient ball}Let $\left[\mathbb{X},\Vert\cdot\Vert_{\mathbb{X}}\right]$
be a quasi-normed vector space and $U\subset\mathbb{X}$ a linear
subspace of $\mathbb{X}$. Then

\[
Q_{U}^{\mathbb{X}}\left(B_{\mathbb{X}}\right)=B_{\mathbb{X}/U}.
\]
\end{lem}

\begin{proof}
Let $\left[x\right]_{U}\in Q_{U}^{\mathbb{X}}\left(B_{\mathbb{X}}\right)$,
then $\left[x\right]_{U}=\left\{ \,x-u;\,u\in U\right\} $ where $x\in B_{\mathbb{X}}$
hence $\Vert x\Vert_{\mathbb{X}}<1$. By Definition \ref{def:Quotientendefinition}
as infimum, we get 

\[
\Vert\left[x\right]_{U}\Vert_{\mathbb{X}/U}=\inf_{u\in U}\Vert x-u\Vert_{\mathbb{X}}\underset{\tiny0\in U}{\leq}\Vert x\Vert_{\mathbb{X}}<1.
\]

On the other hand, if we take $\left[x\right]_{U}\in\mathbb{X}/U$
with $\Vert\left[x\right]_{u}\Vert_{\mathbb{X}/U}<1$ we know that
there exists $x\in\mathbb{X}$, such that $Q_{U}^{\mathbb{X}}x=\left[x\right]_{U}$
and $\Vert x\Vert_{\mathbb{X}}<1$. This yields $\left[x\right]_{U}\in Q_{U}^{\mathbb{X}}\left(B_{\mathbb{X}}\right)$.
\end{proof}

\part{Entropy-, Approximation- and Kolmogorov Numbers}

\section{Entropy Numbers on quasi-Banach Spaces}

In this part we will now introduce the three quantities, beginning
with entropy numbers. There is more than one way to introduce them,
but in the proceeding definition we follow the notation of \cite[Subsect. 1.3.1., Def. 1.]{Edmunds and Triebel},
which is based on dyadic entropy numbers.
\begin{defn}
Let $\mathbb{X}$ and $\mathbb{Y}$ be quasi-Banach spaces, $n\in\mathbb{N}$
and further $T\in\mathcal{L}\left(\mathbb{X},\mathbb{Y}\right)$.
Then we define 

\[
e_{n}\left(T\right):=\inf\left\{ \varepsilon>0\,:\,\exists y_{1},\ldots,y_{2^{n-1}}\in\mathbb{Y}\,:\,T\left(\overline{B}_{\mathbb{X}}\right)\subseteq\bigcup_{i=1}^{2^{n-1}}\left\{ y_{i}+\varepsilon\overline{B}_{\mathbb{Y}}\right\} \right\} 
\]

as the $n$ - th (dyadic) entropy number of the operator $T$.
\end{defn}

\begin{rem}
This is a definition of entropy numbers based on an operator, but
it is also possible to introduce them first on an arbitrary set :

\[
e_{n}\left(A,\mathbb{X}\right):=\inf\left\{ \varepsilon>0\,:\,\exists x_{1},\ldots,x_{2^{n-1}}\in\mathbb{X}\,:\,A\subseteq\bigcup_{i=1}^{2^{n-1}}\left\{ x_{i}+\varepsilon\overline{B}_{\mathbb{X}}\right\} \right\} 
\]

from which the above definition is established through $e_{n}\left(T\right)=e_{n}\left(T\left(\overline{B}_{\mathbb{X}}\right),\mathbb{Y}\right)$.
\end{rem}

The following theorem, though altered in notation to fit in the context
of quasi-Banach spaces, can be found in \cite[Sect. 1.3.]{Carl=000026Stephani}.
The last part of the theorem, which is (C$_{e}$) was slightly modified
taken from \cite[Satz 3.30.]{HaroskeAT}.
\begin{thm}
\label{thm:(Properties-of-e)}(Properties of $e_{n}\left(T\right)$)

Let $\mathbb{X}$,$\mathbb{Y}$ and $\mathbb{W}$ be quasi-Banach
spaces with constants $C_{\mathbb{X}}$,$C_{\mathbb{Y}}$ and $C_{\mathbb{W}}$.
Further let $T\in\mathcal{L}\left(\mathbb{X},\mathbb{Y}\right)$.
\end{thm}

\begin{lyxlist}{00.00.0000}
\item [{(M$_{e}$)}] $C_{\mathbb{Y}}e_{1}\left(T\right)\geq\Vert T\Vert\geq e_{1}\left(T\right)\geq e_{2}\left(T\right)\geq\ldots\geq0$ 
\item [{(A$_{e}$)}] Let $S\in\mathcal{L}\left(\mathbb{X},\mathbb{Y}\right)$,
$n,m\in\mathbb{N}$, then we have \\
$e_{m+n-1}\left(S+T\right)\leq C_{\mathbb{Y}}\left(e_{m}\left(S\right)+e_{n}\left(T\right)\right)$
which is equivalent to \\
$e_{m+n-1}\left(S+T\right)^{\varrho}\leq e_{m}\left(S\right)^{\varrho}+e_{n}\left(T\right)^{\varrho}$
for an equivalent $\varrho$ - norm with \\
$\varrho\in\left(0,1\right]$.
\item [{(P$_{e}$)}] Let $S\in\mathcal{L}\left(\mathbb{Y},\mathbb{W}\right)$,
$n,m\in\mathbb{N}$ , then we have $e_{m+n-1}\left(ST\right)\leq e_{m}\left(S\right)e_{n}\left(T\right)$.
\item [{(C$_{e}$)}] $T\in\mathcal{K}\left(\mathbb{X},\mathbb{Y}\right)\,\Longleftrightarrow\,\lim_{n\rightarrow\infty}e_{n}\left(T\right)=0$
\end{lyxlist}
\begin{proof}
By definition of the entropy numbers as infimum over all $\varepsilon>0$
, the monotonicity (M$_{e}$) is immediately derived, since $\inf_{x\in A}\Vert x\Vert\leq\inf_{x\in B}\Vert x\Vert$
if $B\subseteq A$. To prove $C_{\mathbb{Y}}e_{1}\left(T\right)\geq\Vert T\Vert\geq e_{1}\left(T\right)$,
we show two inequalities. The first is obtained through

\[
T\left(\overline{B}_{\mathbb{X}}\right)\subseteq\Vert T\Vert\overline{B}_{\mathbb{Y}}\underset{\tiny y_{1}=0}{\Longrightarrow}\exists y_{1}\in\mathbb{Y}\,:\,T\left(\overline{B}_{\mathbb{X}}\right)\subseteq\left\{ y_{1}+\Vert T\Vert\overline{B}_{\mathbb{Y}}\right\} 
\]

and taking the infimum over all such $\varepsilon>0$ where for some
$y_{1}\in\mathbb{Y}$, $T\left(\overline{B}_{\mathbb{X}}\right)\subseteq\left\{ y_{1}+\varepsilon\overline{B}_{\mathbb{Y}}\right\} $
is valid and we get $e_{1}\left(T\right)\leq\Vert T\Vert$. Now we
prove the opposite inequality and let $\varepsilon>e_{1}\left(T\right)\geq0$.
This yields that there exists $y_{1}\in\mathbb{Y}$ for which $T\left(\overline{B}_{\mathbb{X}}\right)\subseteq\left\{ y_{1}+\varepsilon\overline{B}_{\mathbb{Y}}\right\} $.
Furthermore for an arbitrary $x\in\overline{B}_{\mathbb{X}}$ there
exist $\eta_{1},\eta_{2}\in\overline{B}_{\mathbb{Y}}$ for which $Tx=y_{1}+\varepsilon\eta_{1}\mbox{ and }T\left(-x\right)=-T\left(x\right)=y_{1}+\varepsilon\eta_{2}$.
Subtracting the second from the first equality yields, 

\[
2Tx=\varepsilon\left(\eta_{1}-\eta_{2}\right)\Longleftrightarrow Tx=\frac{\varepsilon}{2}\left(\eta_{1}-\eta_{2}\right).
\]
 
\[
\Longrightarrow\Vert Tx\Vert_{\mathbb{Y}}\leq\frac{\varepsilon}{2}\Vert\left(\eta_{1}-\eta_{2}\right)\Vert_{\mathbb{Y}}\leq C_{\mathbb{Y}}\frac{\varepsilon}{2}(\underset{\leq1}{\underbrace{\Vert\eta_{1}\Vert_{\mathbb{Y}}}}+\underset{\leq1}{\underbrace{\Vert\eta_{2}\Vert_{\mathbb{Y}})}}\leq C_{\mathbb{Y}}\varepsilon
\]

Because the right-hand side is independent of $x$, the inequality
is maintained, if we take the supremum over all those $x\in\overline{B}_{\mathbb{X}}$.
Hence $\Vert T\Vert\leq C_{\mathbb{Y}}\varepsilon$. Now we can take
the infimum over all $\varepsilon>e_{1}\left(T\right)$ and get $\Vert T\Vert\leq C_{\mathbb{Y}}e_{1}\left(T\right)$.

Let us now have a look at the additivity (A$_{e}$). With given $S$
and $T$, we choose arbitrary $\lambda>e_{n}\left(T\right)$ and $\mu>e_{m}\left(S\right)$.
For these exist $y_{1},\ldots,y_{N},\,z_{1},\ldots,z_{M}$, where
$N\leq2^{n-1}$ and $M\leq2^{m-1}$, such that

\begin{equation}
T\left(\overline{B}_{\mathbb{X}}\right)\subseteq\bigcup_{i=1}^{N}\left\{ y_{i}+\lambda\overline{B}_{\mathbb{Y}}\right\} \,\,\,\mbox{ and }\,\,\,S\left(\overline{B}_{\mathbb{X}}\right)\subseteq\bigcup_{i=1}^{M}\left\{ z_{i}+\mu\overline{B}_{\mathbb{Y}}\right\} .\label{eq:entropy1}
\end{equation}

These inclusions allow us, for any given $x\in\overline{B}_{\mathbb{X}}$
to choose one of the $y_{i}$ and $z_{j}$ for $i\in\left\{ 1,\ldots,N\right\} $,
$j\in\left\{ 1,\ldots,M\right\} $ such that 

\begin{eqnarray*}
Tx & \in & \left\{ y_{i}+\lambda\overline{B}_{\mathbb{Y}}\right\} \,\,\,\mbox{and }\,\,\,Sx\in\left\{ z_{j}+\mu\overline{B}_{\mathbb{Y}}\right\} 
\end{eqnarray*}

Therefore it follows, that for $x\in\overline{B}_{\mathbb{X}}$ exist
$y$,$z\in\overline{B}_{\mathbb{Y}}$ such that $\left(S+T\right)x=y_{i}+z_{j}+\lambda y+\mu z$.
However we find that 
\begin{equation}
\Vert\lambda y+\mu z\Vert_{\mathbb{Y}}\leq C_{\mathbb{Y}}(\lambda\underset{\tiny\leq1}{\underbrace{\Vert y\Vert}}+\mu\underset{\tiny\leq1}{\underbrace{\Vert z\Vert}})\leq C_{\mathbb{Y}}\left(\lambda+\mu\right).\label{eq:Beweismalanders}
\end{equation}

\begin{eqnarray*}
\Longrightarrow & \left(S+T\right)x\in & \left\{ y_{i}+z_{j}+C_{\mathbb{Y}}\left(\lambda+\mu\right)\overline{B}_{\mathbb{Y}}\right\} \\
\overset{\tiny\eqref{eq:entropy1}}{\Longrightarrow} & \left(S+T\right)\overline{B}_{\mathbb{X}}\in & \bigcup_{i=1}^{N}\bigcup_{j=1}^{M}\left\{ y_{i}+z_{j}+C_{\mathbb{Y}}\left(\lambda+\mu\right)\overline{B}_{\mathbb{Y}}\right\} 
\end{eqnarray*}

To obtain the wanted inequality, we need to have a look at the number
of elements in the following set, which is

\begin{equation}
\#\left\{ y_{i}+z_{j},\,i=1,\ldots,N,\,j=1,\ldots,M\right\} \leq NM\leq2^{n-1+m-1}=2^{\left(n+m-1\right)-1}.\label{eq:entropy3}
\end{equation}

\[
\Longrightarrow e_{n+m-1}\left(S+T\right)\leq C_{\mathbb{Y}}\left(\lambda+\mu\right)
\]

And taking the infimum over all those $\lambda$ and $\mu$, we get 

\[
e_{n+m-1}\left(S+T\right)\leq C_{\mathbb{Y}}\left(e_{n}\left(T\right)+e_{m}\left(S\right)\right).
\]
 As we have shown in Theorem \ref{thm:Normequivalence} we can find
an equivalent $\varrho$ - norm, such that $e_{m+n-1}\left(S+T\right)^{\varrho}\leq e_{m}\left(S\right)^{\varrho}+e_{n}\left(T\right)^{\varrho}$.
In particular, we would have in \eqref{eq:Beweismalanders}

\[
\Vert\lambda y+\mu z\Vert_{\mathbb{Y}}^{\varrho}\leq\lambda^{\varrho}\underset{\tiny\leq1}{\underbrace{\Vert y\Vert_{\mathbb{Y}}^{\varrho}}}+\mu^{\varrho}\underset{\tiny\leq1}{\underbrace{\Vert z\Vert_{\mathbb{Y}}^{\varrho}}}\leq\lambda^{\varrho}+\mu^{\varrho}
\]

if we considered a $\varrho$ - norm. The next step would be

\begin{eqnarray*}
\Longrightarrow & \left(S+T\right)x\in & \left\{ y_{i}+z_{j}+\left(\lambda^{\varrho}+\mu^{\varrho}\right)^{1/\varrho}\overline{B}_{\mathbb{Y}}\right\} \\
\overset{\tiny\eqref{eq:entropy1}}{\Longrightarrow} & \left(S+T\right)\overline{B}_{\mathbb{X}}\in & \bigcup_{i=1}^{N}\bigcup_{j=1}^{M}\left\{ y_{i}+z_{j}+\left(\lambda^{\varrho}+\mu^{\varrho}\right)^{1/\varrho}\overline{B}_{\mathbb{Y}}\right\} .
\end{eqnarray*}

By the same arguments as above, we would get

\[
e_{n+m-1}\left(S+T\right)\leq\left(\lambda^{\varrho}+\mu^{\varrho}\right)^{1/\varrho}
\]

and by taking the infimum over all $\lambda$ and $\mu$, we had 

\[
e_{n+m-1}\left(S+T\right)^{\varrho}\leq e_{n}\left(S\right)^{\varrho}+e_{m}\left(T\right)^{\varrho}.
\]

The multiplicativity (P$_{e}$) can be shown by similar arguments.
That is for another given quasi-Banach space $\mathbb{W}$ and operators
$T\in\mathcal{L}\left(\mathbb{X},\mathbb{Y}\right)$ as well as $S\in\mathcal{L}\left(\mathbb{Y},\mathbb{W}\right)$
we can choose $\lambda>e_{n}\left(T\right)$ and $\mu>e_{m}\left(S\right)$,
such that

\begin{equation}
T\left(\overline{B}_{\mathbb{X}}\right)\subseteq\bigcup_{i=1}^{N}\left\{ y_{i}+\lambda\overline{B}_{\mathbb{Y}}\right\} \,\,\,\mbox{and}\,\,\,S\left(\overline{B}_{\mathbb{Y}}\right)\subseteq\bigcup_{j=1}^{M}\left\{ z_{j}+\mu\overline{B}_{\mathbb{W}}\right\} \label{eq:entropy2}
\end{equation}

for $y_{1},\ldots,y_{N}$, $z_{1},\ldots,z_{M}$ and $N<2^{n-1}$,
$M<2^{m-1}$. The right-hand inclusion is equivalent to $S\left(\lambda\overline{B}_{\mathbb{Y}}\right)=\lambda S\left(\overline{B}_{\mathbb{Y}}\right)\subseteq\bigcup_{j=1}^{M}\left\{ \lambda z_{j}+\lambda\mu\overline{B}_{\mathbb{W}}\right\} $,
so that applying the operator $S$ to the left-hand side of \eqref{eq:entropy2}
amounts to 

\[
ST\left(\overline{B}_{\mathbb{X}}\right)\subseteq\bigcup_{i=1}^{N}\bigcup_{j=1}^{M}\left\{ Sy_{i}+\lambda z_{j}+\lambda\mu\overline{B}_{\mathbb{W}}\right\} .
\]

Counting the elements as in \eqref{eq:entropy3} and taking the infimum
over all such $\varepsilon>0$, yields $e_{n+m-1}\left(ST\right)\leq\lambda\mu$.
The last step is taking the infimum over all these $\lambda>e_{n}\left(T\right)$
and $\mu>e_{m}\left(S\right)$. Therefore $e_{n+m-1}\left(ST\right)\leq e_{n}\left(T\right)e_{m}\left(S\right)$.

The last property, which is compactness (C$_{e}$) is immediately
established through the definition of relatively compactness and Remark
\ref{rem:Operator property}. $\lim_{n\rightarrow\infty}e_{n}\left(T\right)=0$
means in particular, that for $\varepsilon>0$ exists $n_{0}\in\mathbb{N}$
such that for all $n\geq n_{0}$ we find that $e_{n}\left(T\right)<\varepsilon$.
Choosing those $y_{1},\ldots,y_{2^{n-1}}$, we have found a finite
$\varepsilon$- net for $T\left(\overline{B}_{\mathbb{X}}\right)$.
This is possible for all $\varepsilon>0$, hence $T\left(\overline{B}_{\mathbb{X}}\right)$
is relatively compact. On the other hand, if $T\left(\overline{B}_{\mathbb{X}}\right)$
is relatively compact, there exists a finite $\varepsilon$ - net
for every $\varepsilon>0$. Since this is only a question of definition,
we can choose only these $\varepsilon$ - nets, which have a dyadic
number of elements. Hence $\lim_{n\rightarrow\infty}e_{n}\left(T\right)=0$. 
\end{proof}

\section{Approximation Numbers on quasi-Banach Spaces}

Let us now define approximation numbers. By doing so, we follow the
notation of \cite[Subsect. 1.3.1., Def. 2.]{Edmunds and Triebel}.
\begin{defn}
\label{def:Approximationszahlen}Let $\mathbb{X}$ and $\mathbb{Y}$
be quasi-Banach spaces and $T\in\mathcal{L}\left(\mathbb{X},\mathbb{Y}\right)$.
For $n\in\mathbb{N}$ we define 

\begin{equation}
a_{n}\left(T\right):=\inf\left\{ \Vert T-S\Vert\,:\,S\in\mathcal{L}\left(\mathbb{X},\mathbb{Y}\right),\,\mbox{rank }S<n\right\} \label{eq:approximation numer}
\end{equation}

as $n$ - th approximation number of the operator $T$.
\end{defn}

As before, the following theorem and its proof in case of Banach spaces
can be found in \cite[Sect. 2.1.]{Carl=000026Stephani} and is adopted
in this bachelor's thesis to fit in the context of quasi Banach spaces.
\begin{thm}
\label{thm:Properties a}(Properties of $a_{n}\left(T\right)$)

Let $\mathbb{X}$,$\mathbb{Y}$ and $\mathbb{W}$ be quasi-Banach
spaces with constants $C_{\mathbb{X}}$,$C_{\mathbb{Y}}$ and $C_{\mathbb{W}}$.
Further let $T\in\mathcal{L}\left(\mathbb{X},\mathbb{Y}\right)$.
\end{thm}

\begin{lyxlist}{00.00.0000}
\item [{(M$_{a}$)}] $\Vert T\Vert=a_{1}\left(T\right)\geq a_{2}\left(T\right)\geq\ldots\geq0$
\item [{(A$_{a}$)}] Let $S\in\mathcal{L}\left(\mathbb{X},\mathbb{Y}\right)$,
$n,m\in\mathbb{N}$, then we have \\
$a_{m+n-1}\left(S+T\right)\leq C_{\mathbb{Y}}\left(a_{m}\left(S\right)+a_{n}\left(T\right)\right)$
which is equivalent to \\
$a_{m+n-1}\left(S+T\right)^{\varrho}\leq a_{m}\left(S\right)^{\varrho}+a_{n}\left(T\right)^{\varrho}$
for an equivalent $\varrho$ - norm with \\
$\varrho\in\left(0,1\right]$.
\item [{(P$_{a}$)}] Let $S\in\mathcal{L}\left(\mathbb{Y},\mathbb{W}\right)$,
$n,m\in\mathbb{N}$ , then we have $a_{m+n-1}\left(ST\right)\leq a_{m}\left(S\right)a_{n}\left(T\right)$.
\item [{(R$_{a}$)}] $\mbox{rank }T<n\,\Longrightarrow\,a_{n}\left(T\right)=0$ 
\item [{(N$_{a}$)}] $\dim\mathbb{X}\geq n\,\Longrightarrow\,a_{n}\left(\mbox{id}_{\mathbb{X}\rightarrow\mathbb{X}}\right)=a_{n}(\mbox{id}_{\mathbb{X}})=1$
\item [{(C$_{a}$)}] $\lim_{n\rightarrow\infty}a_{n}\left(T\right)=0\,\,\,\Longrightarrow\,\,\,T\in\mathcal{K}\left(\mathbb{X},\mathbb{Y}\right)$
\end{lyxlist}
\begin{proof}
The monotonicity (M$_{a}$) is obviously derived, since $\inf_{x\in A}||x||\leq\inf_{x\in B}||x||$
if $B\subseteq A$. Having a closer look at $a_{1}$ we get

\[
a_{1}\left(T\right)=\inf\left\{ \Vert T-S\Vert\,:\,\underset{S\equiv0}{\underbrace{S\in\mathcal{L}\left(\mathbb{X},\mathbb{Y}\right),\,\mbox{rank }S<1}}\right\} =\Vert T\Vert.
\]

To prove the additivity (A$_{a}$) we start with $\lambda>a_{n}\left(T\right)$
and $\mu>a_{m}\left(S\right)$, where $n,m\in\mathbb{N}.$ That means
nothing else, than

\[
\exists L,R\in\mathcal{L}\left(\mathbb{X},\mathbb{Y}\right),\,\mbox{rank }L<n,\,\mbox{rank }R<m\,:\,\Vert T-L\Vert<\lambda\,\,\,\mbox{and}\,\,\,\Vert S-R\Vert<\mu.
\]

Now we define $M:=L+R$. Clearly $M\in\mathcal{L}\left(\mathbb{X},\mathbb{Y}\right)$
and $\mbox{rank }M<n+m-1$. Therefore 

\begin{eqnarray*}
\Vert\left(S+T\right)-M\Vert & = & \sup_{\Vert x\Vert_{\mathbb{X}}=1}\Vert\left[\left(S+T\right)-M\right]x\Vert_{\mathbb{Y}}\\
 & \leq & C_{\mathbb{Y}}\left(\sup_{\Vert x\Vert_{\mathbb{X}}=1}\Vert\left(S-R\right)x\Vert+\sup_{\Vert x\Vert_{\mathbb{X}}=1}\Vert\left(T-L\right)x\Vert\right)\\
 & = & C_{\mathbb{Y}}\left(\Vert S-R\Vert+\Vert T-L\Vert\right)\leq C_{\mathbb{Y}}\left(\lambda+\mu\right).
\end{eqnarray*}

If we now take the infimum over all such operators $M\in\mathcal{L}\left(\mathbb{X},\mathbb{Y}\right)$
with $\mbox{rank }M<n+m-1$, we get $a_{n+m-1}\left(S+T\right)<C_{\mathbb{Y}}\left(\lambda+\mu\right)$.
Taking the infimum over $\lambda$ and $\mu$ amounts to $a_{n+m-1}\left(S+T\right)\leq C_{\mathbb{Y}}\left(a_{n}\left(T\right)+a_{m}\left(S\right)\right)$
, which is again equivalent to $a_{n+m-1}\left(S+T\right)^{\varrho}\leq a_{n}\left(S\right)^{\varrho}+a_{m}\left(T\right)^{\varrho}$.
We will not prove the equivalence, since the idea is the same, as
we have seen in (A$_{e}$) of Theorem \ref{thm:(Properties-of-e)}.

The multiplicativity (P$_{a}$) is established through a similar proof,
in which we let $\lambda>a_{n}\left(T\right)$ and $\mu>a_{m}\left(S\right)$
for given $S\in\mathcal{L}\left(\mathbb{Y},\mathbb{W}\right)$ and
$n,m\in\mathbb{N}$. As above we get 

\begin{eqnarray*}
 & \exists L\in & \mathcal{L}\left(\mathbb{X},\mathbb{Y}\right),\,\mbox{rank }L<n\,:\,\Vert T-L\Vert<\lambda\\
\mbox{and} & \exists R\in & \mathcal{L}\left(\mathbb{Y},\mathbb{W}\right),\,\mbox{rank }R<m\,:\,\Vert S-R\Vert<\mu.
\end{eqnarray*}

We go on by defining $M:=RT+SL-RL\in\mathcal{L}\left(\mathbb{X},\mathbb{W}\right)$.
Hence

\begin{eqnarray*}
\Vert ST-M\Vert & = & \Vert ST-RT-SL+RL\Vert=\Vert\left(S-R\right)\left(T-L\right)\Vert\\
 & \leq & \Vert S-R\Vert\Vert T-L\Vert<\lambda\mu.
\end{eqnarray*}

Furthermore $\mbox{rank }M\leq\mbox{rank }SL+\mbox{rank }\left(R\left(T-L\right)\right)\leq\mbox{rank }L+\mbox{rank }R<n+m-1$.
Knowing, that there exists such an operator, we can take the infimum
over these, which amounts to $a_{n+m-1}\left(ST\right)\leq\lambda\mu$.
Taking the infimum over $\lambda$ and $\mu$ yields 
\[
a_{n+m-1}\left(ST\right)\leq a_{n}\left(S\right)a_{m}\left(T\right)
\]

The rank property (R$_{a}$) is quickly derived, because of $\mbox{rank }T<n$,
it is a fair competitor for the infimum, which results to

\[
0\leq a_{n}\left(T\right)=\inf\left\{ \Vert T-S\Vert:\,S\in\mathcal{L}\left(\mathbb{X},\mathbb{Y}\right),\mbox{rank }S<n\right\} \leq\Vert T-T\Vert=0.
\]

Next we will prove (N$_{a}$) for which the monotonicity is needed
in $a_{n}\left(\mbox{id}_{\mathbb{X}}\right)\leq a_{1}\left(\mbox{id}_{\mathbb{X}}\right)=\Vert\mbox{id}_{\mathbb{X}}\Vert=1$.
If we can show, that $a_{n}\left(\mbox{id}_{\mathbb{X}}\right)\geq1$
the equality is established. For that, let $\dim\mathbb{X}\geq n$
and $L\in\mathcal{L}\left(\mathbb{X},\mathbb{X}\right)$ and $\mbox{rank }L<n$.
Hence there exists $x_{0}\in\mathbb{X},\,x_{0}\neq0$ for which $Lx_{0}=0$.
Without loss of generality, we can say, that $\Vert x_{0}\Vert_{\mathbb{X}}=1$
(otherwise, we scale it).

\[
\Longrightarrow1=\Vert x_{0}\Vert_{\mathbb{X}}=\Vert x_{0}-\underset{=0}{\underbrace{Lx_{0}}}\Vert_{\mathbb{X}}\leq\sup_{\Vert x\Vert_{\mathbb{X}}=1}\Vert\left(\mbox{id}_{\mathbb{X}}-L\right)x\Vert_{\mathbb{X}}=\Vert\mbox{id}_{\mathbb{X}}-L\Vert
\]

If we take the infimum over all such $L$, we end up on \eqref{eq:approximation numer},
which is the definition of the $n$ - th approximation number. Hence
$a_{n}\left(\mbox{id}_{\mathbb{X}}\right)\geq1$ and with our first
step, the equality is proven.

The last property, which is compactness (C$_{a}$), is immediately
given, because $T\in\mathcal{L}\left(\mathbb{X},\mathbb{Y}\right)$
and $\lim_{n\rightarrow\infty}a_{n}\left(T\right)=0$. This means,
that there exists a sequence of finite rank operators, that converges
to $T$. Hence $T\in\mathcal{K}\left(\mathbb{X},\mathbb{Y}\right)$
(See Remark \ref{rem:Operator property})
\end{proof}
\begin{rem}
\label{Remark of Equivalnce}We have seen that property (C$_{e}$)
of Theorem \ref{thm:(Properties-of-e)} is an equivalence, whereas
(C$_{a}$) of Theorem \ref{thm:Properties a} is only an implication.
We should point out, that we have no loss of information when we switch
from Banach spaces to quasi-Banach spaces, which means that the opposite
implication is not even valid in the Banach space case.

We do however have a loss of information when considering the quasi-Banach
space case concerning (R$_{a}$), since we have an equivalence in
the Banach space case. (See \cite[Sect. 2.4., A4]{Carl=000026Stephani}.)
\end{rem}

\section{Kolmogorov Numbers on quasi-Banach Spaces}

At last we introduce Kolmogorov numbers. We will follow the notation
of \cite[Abschn. 3.3.]{HaroskeAT} in this part.
\begin{defn}
Let $\mathbb{X}$ and $\mathbb{Y}$ be quasi-Banach spaces and $T\in\mathcal{L}\left(\mathbb{X},\mathbb{Y}\right)$.
For $n\in\mathbb{N}$ we call

\[
d_{n}\left(T\right)=\inf_{U_{n}\subset\mathbb{Y},\,\dim U_{n}<n}\sup_{||x||_{\mathbb{X}}\leq1}\inf_{y\in U_{n}}||Tx-y||_{\mathbb{Y}}
\]

the $n$-th Kolmogorov number of the operator $T$.
\end{defn}

\begin{rem}
\label{rem:remark-proof}This is again a definition, which is based
on an operator $T$. If $\mathbb{X}$ is a quasi-Banach space and
$A\subset\mathbb{X}$, $n\in\mathbb{N}$ then Kolmogorov numbers can
also be introduced as 

\[
d_{n}(A,\mathbb{X}):=\inf_{U_{n}\subset\mathbb{X},\,\dim U_{n}<n}\sup_{x\in A}\inf_{y\in U_{n}}||x-y||_{\mathbb{X}}
\]

from which the above number is established through $d_{n}(T)=d_{n}(T(\overline{B}_{\mathbb{X}}),\mathbb{Y})$.
According to this definition of Kolmogorov numbers based on sets,
we can make the following statement:

Let $\mathbb{X}$ be a quasi-Banach space with $\dim\mathbb{X}\geq n$
for $n\in\mathbb{N}$. Then

\[
d_{k}\left(\overline{B}_{\mathbb{X}},\mathbb{X}\right)=1,\,\mbox{for }k=1,\ldots,n.
\]
\end{rem}

\begin{proof}
At first we notice $d_{1}\left(\overline{B}_{\mathbb{X}},\mathbb{X}\right)=\sup_{x\in\overline{B}_{\mathbb{X}}}\Vert x\Vert=1$.
Furthermore we can easily see, that $d_{k}\left(\overline{B}_{\mathbb{X}},\mathbb{X}\right)\leq d_{m}\left(\overline{B}_{\mathbb{X}},\mathbb{X}\right)$
if $k\geq m$. This is because 

\begin{eqnarray*}
d_{m}(\overline{B}_{\mathbb{X}},\mathbb{X}) & = & \inf_{U_{m}\subset\mathbb{X},\,\dim U_{m}<m}\sup_{||x||_{\mathbb{X}}\leq1}\inf_{y\in U_{m}}\Vert x-y\Vert_{\mathbb{X}}\\
 & \geq & \inf_{U_{m+1}\subset\mathbb{X},\,\dim U_{m+1}<m+1}\sup_{||x||_{\mathbb{X}}\leq1}\inf_{y\in U_{m+1}}\Vert x-y\Vert_{\mathbb{X}}=d_{m+1}(T)
\end{eqnarray*}

and of course $\inf_{x\in A}||x||\leq\inf_{x\in B}||x||$ if $B\subseteq A$.
Hence 

\[
d_{n}\left(\overline{B}_{\mathbb{X}},\mathbb{X}\right)\leq d_{1}\left(\overline{B}_{\mathbb{X}},\mathbb{X}\right)=1.
\]

Now we need to show that $d_{n}\left(\overline{B}_{\mathbb{X}},\mathbb{X}\right)\geq1$.
At first we will clarify that for all such subspaces $U_{n}\subset\mathbb{X},$
with $\dim U_{n}<n$ there exists $x_{n}\in\mathbb{X},\,x_{n}\neq0$
such that $\inf_{y\in U}\Vert y-x_{n}\Vert_{\mathbb{X}}=\Vert x_{n}\Vert_{\mathbb{X}}$.
The case $x_{n}\in U_{n}$ is obvious, so for a given subspace $U_{n}\subset\mathbb{X}$,
we choose an arbitrary $\xi\in\mathbb{X}\backslash U_{n}$. With Theorem
\ref{thm:bestapproximation} we know, that there exists a best approximation.
This means that there exists $y_{n}\in U_{n}$, such that $0<\Vert\xi-y_{n}\Vert_{\mathbb{X}}=\inf_{u\in U_{n}}\Vert\xi-u\Vert_{\mathbb{X}}$
. Hence by defining $x_{n}:=\xi-y_{n}\in\mathbb{X}$, we get $x_{n}\neq0$
and 

\begin{eqnarray}
\Vert x_{n}\Vert_{\mathbb{X}} & = & \Vert\xi-y_{n}\Vert_{\mathbb{X}}=\inf_{u\in U_{n}}\Vert\xi-y_{n}-\left(u-y_{n}\right)\Vert_{\mathbb{X}}=\inf_{u\in U_{n}}\Vert x_{n}-\underset{=:y\in U_{n}}{\underbrace{\left(u-y_{n}\right)}}\Vert_{\mathbb{X}}\nonumber \\
 & = & \inf_{y\in U_{n}}\Vert x_{n}-y\Vert_{\mathbb{X}}.\label{eq:chain}
\end{eqnarray}

Without loss of generality we can say $\Vert x_{n}\Vert=1$ (otherwise,
we could scale it). This yields 

\[
\sup_{x\in\overline{B}_{\mathbb{X}}}\inf_{y\in U_{n}}\Vert y-x\Vert\geq\inf_{y\in U_{n}}\Vert y-x_{n}\Vert\overset{\tiny\eqref{eq:chain}}{=}\Vert x_{n}\Vert=1.
\]

Thus taking the infimum over all such spaces $U_{n}$, we get $d_{n}\left(\overline{B}_{\mathbb{X}},\mathbb{X}\right)\geq1$
. Together with step 1, this yields the equality.
\end{proof}
As before, we will now show some basic properties of Kolmogorov numbers,
which can be found in \cite[Satz 3.28.]{HaroskeAT} for the case of
Banach spaces and which are slightly modified here.
\begin{thm}
\label{thm:Properties d}(Properties of $d_{n}(T)$)

Let $\mathbb{X}$,$\mathbb{Y}$ and $\mathbb{W}$ be quasi-Banach
spaces, with constants $C_{\mathbb{X}}$, $C_{\mathbb{Y}}$ and $C_{\mathbb{W}}$.
Further let $T\in\mathcal{L}\left(\mathbb{X},\mathbb{Y}\right)$. 
\end{thm}

\begin{lyxlist}{00.00.0000}
\item [{($\mbox{M}_{d}$)}] $\Vert T\Vert=d_{1}(T)\geq d_{2}(T)\geq\ldots\geq0$
\item [{($\mbox{A}_{d}$)}] Let $S\in\mathcal{L}\left(\mathbb{X},\mathbb{Y}\right)$,
$n,\,m\in\mathbb{N}$, then $d_{m+n-1}\left(S+T\right)\leq C_{\mathbb{Y}}\left(d_{m}\left(S\right)+d_{n}\left(T\right)\right)$
which is equivalent to $d_{m+n-1}\left(S+T\right)^{\varrho}\leq d_{m}\left(S\right)^{\varrho}+d_{n}\left(T\right)^{\varrho}$
for an equivalent $\varrho$ - norm with $\varrho\in\left(0,1\right]$.
\item [{($\mbox{P}_{d}$)}] Let $S\in\mathcal{L}\left(\mathbb{Y},\mathbb{W}\right)$,
$n,\,m\in\mathbb{N}$. Then $d_{m+n-1}\left(ST\right)\leq d_{m}\left(S\right)d_{n}(T)$.
\item [{($\mbox{R}_{d}$)}] $\mbox{rank\,}T<n\,\Longrightarrow\,d_{n}\left(T\right)=0$
\item [{($\mbox{N}_{d}$)}] $\dim\mathbb{X}\geq n\,\Longrightarrow\,d_{n}\left(\mbox{id}_{\mathbb{X}\rightarrow\mathbb{X}}\right)=d_{n}(\mbox{id}_{\mathbb{X}})=1$
\item [{($\mbox{C}_{d}$)}] $T\in\mathcal{K}(\mathbb{X},\mathbb{Y})\Longleftrightarrow\lim_{n\rightarrow\infty}d_{n}(T)=0$
\end{lyxlist}
\begin{proof}
The proof of monotonicity (M$_{d}$) is similar to Remark \ref{rem:remark-proof}
. To prove the second fact, we take a look at an arbitrary operator
$T\in\mathcal{L}(\mathbb{X},\mathbb{Y})$. This yields that $T\left(\overline{B}_{\mathbb{X}}\right)$
is bounded, and further

\[
d_{1}\left(T\right)=\inf_{U_{1}\subset\mathbb{Y},\,\dim U_{1}<1}\sup_{||x||_{\mathbb{X}}\leq1}\inf_{y\in U_{1}}\Vert Tx-y\Vert_{\mathbb{Y}}=\sup_{||x||_{\mathbb{X}}\leq1}\Vert Tx\Vert_{\mathbb{Y}}=\Vert T\Vert_{\mathbb{Y}}.
\]

To prove the additivity (A$_{d}$), let $\varepsilon>0$, $n,\,m\in\mathbb{N}$.
By definition of $d_{n}$ as infimum over all subspaces $U_{n}\subset\mathbb{Y}$
with $\dim U_{n}<n$, we gain the following:

\[
\exists U_{m}\subset\mathbb{Y},\,\dim U_{m}<m\mbox{ and }V_{n}\subset\mathbb{Y},\,\dim V_{n}<n\,\forall x\in\overline{B}_{\mathbb{X}}\,\exists u_{m}^{x}\in U_{m},\,v_{n}^{x}\in V_{n}:
\]

\[
\Vert Sx-u_{m}^{x}\Vert_{\mathbb{Y}}<d_{m}(S)+\varepsilon\mbox{\,\,\,\,\,\ and\,\,\,\,\,}\Vert Tx-v_{n}^{x}\Vert_{\mathbb{Y}}<d_{n}(T)+\varepsilon
\]

Now we denote $W_{m,n}=U_{m}+V_{n}\subset\mathbb{Y}$. We notice that
$\dim W_{m,n}<n+m-1$. As above, we can see that for all $x\in\overline{B}_{\mathbb{X}}$
there exists $w_{m,n}^{x}=u_{m}^{x}+v_{n}^{x}\in W_{m,n}$ , such
that

\begin{eqnarray}
\Vert(S+T)x-w_{m,n}^{x}\Vert_{\mathbb{Y}} & = & \Vert Sx-u_{m}^{x}+Tx-v_{n}^{x}\Vert_{\mathbb{Y}}\nonumber \\
 & \leq & C_{\mathbb{Y}}\left(\Vert Sx-u_{m}^{x}\Vert_{\mathbb{Y}}+\Vert Tx-v_{n}^{x}\Vert_{\mathbb{Y}}\right)\label{eq:AdditivityD_n}\\
 & < & C_{\mathbb{Y}}\left(d_{m}(S)+d_{n}(T)+2\varepsilon\right).\nonumber 
\end{eqnarray}

The inequality \eqref{eq:AdditivityD_n} is maintained if we take
the supremum of all $x\in\overline{B}_{\mathbb{X}}$ over the infimum
of all such $w_{m,n}\in W_{m,n}$. Because there exist such $W_{m,n}$
we can also take the infimum over all such subspaces $W_{m,n}\subset\mathbb{Y}$
with $\dim W_{m,n}<n+m-1$. Since $\varepsilon$ was arbitrary, we
let $\varepsilon\rightarrow0$ and gain the additivity (A$_{d}$)

\[
d_{n+m-1}(S+T)\leq C_{\mathbb{Y}}\left(d_{m}(S)+d_{n}(T)\right).
\]

As we already have established two times before, this is equivalent
to 

\[
d_{n+m-1}\left(S+T\right)^{\varrho}\leq d_{m}\left(S\right)^{\varrho}+d_{n}\left(T\right)^{\varrho}
\]

and is left here unproven, since the idea is the same as in (A$_{e}$)
of Theorem \ref{thm:(Properties-of-e)}.

We advance with the multiplicativity (M$_{d}$) and start the same
way as above by $\varepsilon>0$ and $n,m\in\mathbb{N}$. Also with
the same arguments as above, through definition of the infimum, we
have

\begin{eqnarray}
 &  & \exists U_{m}\subset\mathbb{Y},\,\dim U_{m}<m\,\,\forall x\in\overline{B}_{\mathbb{X}}\,\,\exists u_{m}^{x}\in U_{m}:\,||Tx-u_{m}^{x}||_{\mathbb{Y}}\leq d_{m}\left(T\right)+\varepsilon,\label{eq:Theorem2a}\\
\mbox{} & \mbox{and} & \exists V_{n}\subset\mathbb{W},\,\dim V_{n}<n\,\,\forall y\in\overline{B}_{\mathbb{Y}}\,\,\exists v_{n}^{y}\in V_{n}:\,||Sy-v_{n}^{y}||_{\mathbb{W}}\leq d_{n}\left(S\right)+\varepsilon.\label{eq:Theorem2b}
\end{eqnarray}

Now let $x\in\overline{B}_{\mathbb{X}}.$ With our first premise \eqref{eq:Theorem2a},
we gain 

\[
\left\Vert \frac{Tx-u_{m}^{x}}{d_{m}(T)+\varepsilon}\right\Vert <1\mbox{ \,\,\ and define\,\,}y=y(x):=\frac{Tx-u_{m}^{x}}{d_{m}(T)+\varepsilon}.
\]

$\Longrightarrow\exists w_{m,n}^{x}=Su_{m}^{x}+\left(d_{m}\left(T\right)+\varepsilon\right)v_{n}^{y(x)}\in S\left(U_{m}\right)+V_{n}=W_{m,n}\subset\mathbb{W}$

\begin{eqnarray*}
 & \left\Vert S\underset{y(x)}{\underbrace{\left(\frac{Tx-u_{m}^{x}}{d_{m}(T)+\varepsilon}\right)}}-\underset{v_{n}^{y(x)}}{\underbrace{\frac{w_{m,n}^{x}-Su_{m}^{x}}{d_{m}\left(T\right)+\varepsilon}}}\right\Vert _{\mathbb{W}}\overset{\tiny\eqref{eq:Theorem2b}}{\text{<}}d_{n}\left(S\right)+\varepsilon\\
\Longleftrightarrow & \left\Vert STx-w_{m,n}^{x}\right\Vert _{\mathbb{W}}\leq\left(d_{n}\left(S\right)+\varepsilon\right)\left(d_{m}\left(T\right)+\varepsilon\right)
\end{eqnarray*}

As above, this inequality is of course maintained, if we take the
supremum over all $x\in\overline{B}_{\mathbb{X}}$ over the infimum
of all such $w_{m,n}$ in this specific $W_{m,n}$. Since such $W_{m,n}\subset\mathbb{W}$
exist and have $\dim W_{m,n}<m+n-1$, we can take the infimum over
them. With $\varepsilon$ arbitrary, we let $\varepsilon\downarrow0$
and the desired statement follows as

\[
d_{m+n-1}\left(ST\right)\leq d_{n}\left(S\right)d_{m}\left(T\right).
\]

The statement, which denotes rank-properties (R$_{d}$) with respect
to Kolmogorov numbers is easily obtained, since $\mbox{rank }T=\dim\mathcal{R}\left(T\right)<n$,
we simply put $U_{n}:=\mathcal{R}\left(T\right)$ and get $d_{n}\left(T\right)=0$. 

For proving property (N$_{d}$), we simply use Remark \ref{rem:remark-proof}
and the fact, that $T:=\mbox{id : }\mathbb{X}\longrightarrow\mathbb{X}\in\mathcal{L}\left(\mathbb{X},\mathbb{X}\right)$.
Hence

\[
d_{k}\left(T\right)=d_{k}\left(T\left(\overline{B}_{\mathbb{X}}\right),\mathbb{X}\right)=d_{k}\left(\overline{B}_{\mathbb{X}},\mathbb{X}\right)=1\,\mbox{\,\,\,\ for\,\,\,\ }k=1,\ldots,n.
\]

At last we will have a look at compactness properties (C$_{d}$).
For the first direction, let $T\in\mathcal{K}\left(\mathbb{X},\mathbb{Y}\right)$.
Then with Remark \ref{rem:Operator property} we know that $T\left(\overline{B}_{\mathbb{X}}\right)$
is relatively compact, hence we can find a finite $\varepsilon$ -
net. That means 

\[
\exists n_{0}\in\mathbb{N},\,x_{0},\ldots,x_{n_{0}}\in\mathbb{X}\,\forall x\in T\left(\overline{B}_{\mathbb{X}}\right):\,\min_{i=1,\ldots,n_{0}}\Vert x-x_{i}\Vert_{\mathbb{X}}\leq\varepsilon.
\]

We define $U_{n_{0}}:=\mbox{span }\left\{ x_{1},\ldots,x_{n_{0}}\right\} $
and since $\dim U_{n_{0}}\leq n_{0}$ we can derive that $d_{n_{0}+1}\left(T\right)\leq\varepsilon$.
Using the monotonicity (M$_{d}$) we get

\[
\exists n_{0}\in\mathbb{N}\,\forall n>n_{0}:\,d_{n}\left(T\right)\leq\varepsilon\,\Longleftrightarrow\,\lim_{n\rightarrow\infty}d_{n}\left(T\right)=0.
\]

To prove the opposite direction, we suppose that $\lim_{n\rightarrow\infty}d_{n}\left(T\right)=0$.
$T\in\mathcal{L}\left(\mathbb{X},\mathbb{Y}\right)$ yields, that
$T\left(\overline{B}_{\mathbb{X}}\right)$ is bounded in $\mathbb{Y}.$

\[
\Longrightarrow d_{1}\left(T\right)=\sup_{x\in T\left(\overline{B}_{\mathbb{X}}\right)}\Vert x\Vert_{\mathbb{X}}<\infty
\]

If we have a look at our premise $\lim_{n\rightarrow\infty}d_{n}\left(T\right)=0$,
we see that this means

\[
\forall\varepsilon>0\,\exists n_{0}\in\mathbb{N\,}\forall n>n_{0}\,\exists U_{n}\subset\mathbb{Y},\dim U_{n}<n\,\forall x\in T\left(\overline{B}_{\mathbb{X}}\right)\,\exists u_{n}^{x}\in U_{n}:\,\Vert x-u_{n}^{x}\Vert_{\mathbb{Y}}<\varepsilon.
\]

For these $u_{n}^{x}\in U_{n}$, we have

\[
\Vert u_{n}^{x}\Vert_{\mathbb{Y}}\leq C_{\mathbb{Y}}\left(\Vert x-u_{n}^{x}\Vert_{\mathbb{Y}}+\Vert x\Vert_{\mathbb{Y}}\right)\leq C_{\mathbb{Y}}\left(\varepsilon+\sup_{x\in T\left(\overline{B}_{\mathbb{X}}\right)}\Vert x\Vert_{\mathbb{Y}}\right)=C_{\mathbb{Y}}\left(\varepsilon+d_{1}\left(T\right)\right).
\]

If we define $M_{0}:=\left\{ u\in U_{n}:\,\Vert u\Vert_{\mathbb{Y}}\leq C_{\mathbb{Y}}\left(\varepsilon+d_{1}\left(T\right)\right)\right\} $,
we see that $M_{0}$ is bounded and $u_{n}^{x}\in M_{0}\subset U_{n}$.
$\dim U_{n}<n$ yields that $M_{0}$ is relatively compact and therefore
we can find a finite $\varepsilon$ - net $M_{1}:=\left\{ \eta_{0},\ldots,\eta_{m}\right\} $
for $M_{0}$. This specific $M_{1}$ is a $2C_{\mathbb{Y}}\varepsilon$
- net for $T\left(\overline{B}_{\mathbb{X}}\right)$, hence $T\left(\overline{B}_{\mathbb{X}}\right)$
is relatively compact and that means, that $T\in\mathcal{K}\left(\mathbb{X},\mathbb{Y}\right)$. 
\end{proof}
\begin{rem}
As we have already done in Remark \ref{Remark of Equivalnce} for
approximation numbers, we should point out, that we have a loss of
information, when considering quasi-Banach spaces in property (R$_{d}$),
since we have an equivalence for Banach spaces. (See \cite[Satz 3.28.]{HaroskeAT}.)
\end{rem}

We have established various properties of Kolmogorov numbers and have
already seen, that they can be introduced in different ways (on arbitrary
sets or operators), and we will now show the equivalence to other
definitions taken from \cite[Sect. 2.2.]{Carl=000026Stephani} and
\cite[Ch. 2., 2.5.2.]{Albrecht Pietsch}.
\begin{prop}
\label{pro:alt kolmogorov}Let $\mathbb{X}$ and $\mathbb{Y}$ be
arbitrary quasi-Banach spaces and $T\in\mathcal{L}\left(\mathbb{X},\mathbb{Y}\right)$,
$n\in\mathbb{N}$. Then the $n$ - th Kolmogorov number $d_{n}\left(T\right)$
can be expressed as 

\begin{eqnarray*}
\mbox{(i) } & d_{n}\left(T\right) & =\inf\left\{ \varepsilon>0\,:\,T\left(\overline{B}_{\mathbb{X}}\right)\subset N_{\varepsilon}+\varepsilon\overline{B}_{\mathbb{Y}},\,N_{\varepsilon}\subset\mathbb{Y},\,\dim N_{\varepsilon}<n\right\} .\\
\mbox{(ii)} & d_{n}\left(T\right) & =\inf\left\{ \Vert Q_{V}^{\mathbb{Y}}T\Vert\,:\,V\subset\mathbb{Y},\,\dim V<n\right\} .
\end{eqnarray*}
\end{prop}

\begin{proof}
In the \emph{first step}, we show the equivalence to a definition
taken by \cite{Carl=000026Stephani}. Again we will use Theorem \ref{thm:Normequivalence}
to consider $\varrho$ - norms instead of quasi-norms. Let

\[
\hat{d_{n}}\left(T\right):=\inf\left\{ \varepsilon>0\,:\,T\left(\overline{B}_{\mathbb{X}}\right)\subset N_{\varepsilon}+\varepsilon\overline{B}_{\mathbb{Y}},\,N_{\varepsilon}\subset\mathbb{Y},\,\dim N_{\varepsilon}<n\right\} .
\]

We now want to show that $\hat{d_{n}}\left(T\right)\leq d_{n}\left(T\right)$.
By definition, we can find a subspace $N\subset\mathbb{Y}$, $\dim N<n$
such that

\[
\sup_{\Vert x\Vert\leq1}\inf_{y\in N}\Vert Tx-y\Vert\leq d_{n}\left(T\right)+\delta
\]

for some $\delta>0$. That is to say, that for every element $x\in\overline{B}_{\mathbb{X}}$
exists a $y\in N$, such that

\[
Tx\in\left\{ y+\left(d_{n}\left(T\right)^{\varrho}+\delta^{\varrho}\right)^{1/\varrho}\overline{B}_{\mathbb{Y}}\right\} ,
\]

which means nothing more than

\[
T\left(\overline{B}_{\mathbb{X}}\right)\subset N+\left(d_{n}\left(T\right)^{\varrho}+\delta^{\varrho}\right)^{1/\varrho}\overline{B}_{\mathbb{Y}}.
\]

Hence be letting $\delta\downarrow0$ we get $\hat{d_{n}}\left(T\right)\leq d_{n}\left(T\right).$
Now we establish the opposite inequality. For $\delta>0$ we choose
$N\subset\mathbb{Y}$, $\dim N<n$ such that

\[
T\left(\overline{B}_{\mathbb{X}}\right)\subset N+\left(\hat{d_{n}}\left(T\right)+\delta\right)\overline{B}_{\mathbb{Y}},
\]

which means that for every $x\in\overline{B}_{\mathbb{X}}$ exists
$y\in N$, such that

\[
Tx\in\left\{ y+\left(\hat{d_{n}}\left(T\right)+\delta\right)\overline{B}_{\mathbb{Y}}\right\} \Longleftrightarrow\Vert Tx-y\Vert\leq\left(\hat{d_{n}}\left(T\right)+\delta\right).
\]

Of course this stays correct if we take the supremum over all $x\in\overline{B}_{\mathbb{X}}$
over the infimum of all $y\in N$. Hence by taking the infimum over
all such subspaces and letting $\delta\downarrow0$, we get $d_{n}\left(T\right)\leq\hat{d_{n}}\left(T\right)$. 

In our \emph{second step} we show the equivalence of (i) to (ii),
where the definition in (ii) is taken from \cite[Ch. 2., 2.5.2.]{Albrecht Pietsch}.
So let $d_{n}\left(T\right)=\varepsilon$. According to the equivalent
definition of the $n$ - th Kolmogorov number established in (i),
we can take a subspace $V\subset\mathbb{Y}$, $V=V\left(\varepsilon\right)$
with $\dim V<n$ and get

\[
T\left(\overline{B}_{\mathbb{X}}\right)\subseteq V+\varepsilon\overline{B}_{\mathbb{X}}.
\]

By applying the quotient map of Definition \ref{def:Quotientendefinition}
we get

\[
Q_{V}^{\mathbb{Y}}T\left(\overline{B}_{\mathbb{X}}\right)\subseteq\varepsilon Q_{V}^{\mathbb{Y}}\left(\overline{B}_{\mathbb{X}}\right)=\varepsilon\overline{B}_{\mathbb{Y}/V}.
\]

Hence $\Vert Q_{V}^{\mathbb{Y}}T\Vert\leq\varepsilon$ and taking
the infimum over all such subspaces $V$, and the notation $\tilde{d_{n}}\left(T\right)=\inf\left\{ \Vert Q_{V}^{\mathbb{Y}}T\Vert:\,V\subset\mathbb{Y},\,\dim V<n\right\} $,
we get $\tilde{d_{n}}\left(T\right)\leq d_{n}\left(T\right)$.

We now wish to show that $\tilde{d_{n}}\left(T\right)\geq d_{n}\left(T\right)$.
We choose a $\delta>0$ and an arbitrary subspace $V\subset\mathbb{Y}$
with $\dim V<n$, such that $\Vert Q_{V}^{\mathbb{Y}}T\Vert<\tilde{d_{n}}\left(T\right)+\delta$.
This inequality implies

\[
Q_{V}^{\mathbb{Y}}T\left(\overline{B}_{\mathbb{X}}\right)\subseteq\left(\tilde{d_{n}}\left(T\right)+\delta\right)\overline{B}_{\mathbb{Y}/V}\overset{\tiny\mbox{Lemma}\ref{lem:quotient ball}}{=}Q_{V}^{\mathbb{Y}}\left(\left(\tilde{d_{n}}\left(T\right)+\delta\right)\overline{B}_{\mathbb{Y}}\right).
\]

If we pass on to the inverse (as far as sets are concerned), we have 

\[
V+T\left(\overline{B}_{\mathbb{X}}\right)\subseteq V+\left(\tilde{d_{n}}\left(T\right)+\delta\right)\overline{B}_{\mathbb{Y}}.
\]

\[
\Longrightarrow T\left(\overline{B}_{\mathbb{X}}\right)\subseteq V+\left(\tilde{d_{n}}\left(T\right)+\delta\right)\overline{B}_{\mathbb{Y}}
\]

Hence, with (i) we find that $d_{n}\left(T\right)\leq\tilde{d_{n}}\left(T\right)+\delta$.
If we let $\delta\downarrow0$, we finally get $d_{n}\left(T\right)\leq\tilde{d_{n}}\left(T\right)$,
hence the equality is established.
\end{proof}
Now that we have established the three numbers with respect to operators,
there arises the question of a connection between them, on which we
will have a look at in the following part.

\part{Relationship between the Numbers}

\section{Relationship between Entropy and Approximation Numbers}

To begin with the investigation of connections between the numbers,
we shall start with approximation- and entropy numbers. We will compare
both concerning mean values and their limits. Since entropy numbers
hold more information about the structure of the operator, one could
identify them with the modulus of continuity, whereas approximation
numbers hold more information about how good an approximation can
be carried out and thus could be identified with the best approximation
of continuous functions by polynomials. Hence the idea of this part
is to give inequalities of Bernstein-Jackson type for operators and
to introduce the theory of $s$-scales. 
\begin{lem}
\label{lem:Estimation unproved}An operator $T$ acting between real
quasi-Banach spaces $\mathbb{X}$ and $\mathbb{Y}$ is of $\mbox{rank }m$
if and only if there exist constants $C,c>0$ such that

\[
c\cdot2^{-\left(n-1\right)/m}\leq e_{n}\left(T\right)\leq C\Vert T\Vert\cdot2^{-\left(n-1\right)/m},\,\,\mbox{for }n=1,2,3,\ldots.
\]

If $T$ acts between complex quasi-Banach spaces $\mathbb{X}$ and
$\mathbb{Y}$, it is of $\mbox{rank }m$ if and only if there exists
a constant $C,c>0$ such that

\[
c\mbox{\ensuremath{\cdot}}2^{-\left(n-1\right)/2m}\leq e_{n}\left(T\right)\leq C\Vert T\Vert\cdot2^{-\left(n-1\right)/2m},\,\,\mbox{for }n=1,2,3,\ldots.
\]
\end{lem}

\begin{proof}
This statement is taken from \cite[p. 21]{Carl=000026Stephani} and
is slightly altered here for the context of quasi-Banach spaces. It
is used here without proof, but its validity can be understood, by
reconstructing \cite[Satz 3.31]{HaroskeAT} for general linear operators
acting between quasi-Banach spaces. (See also \cite[Üb. III-4]{HaroskeAT}.)
\end{proof}
With this lemma in mind, we can state the following proposition, which
will become useful later on. It can be found in \cite[Sect. 3.1., Thm. 3.1.1.]{Carl=000026Stephani}
for the case of Banach spaces and is slightly modified here, to fit
in the context of quasi-Banach spaces. A similar theorem, but more
general, as well as its proof can also be found in \cite[Subsect. 1.3.3.]{Edmunds and Triebel}.
\begin{prop}
\label{pro:supremumsdinge}Let $0<p<\infty$ and let $T\in\mathcal{L}\left(\mathbb{X},\mathbb{Y}\right)$,
where $\mathbb{X}$ and $\mathbb{Y}$ are arbitrary quasi-Banach spaces
with constants $C_{\mathbb{X}}$ and $C_{\mathbb{Y}}$. Then

\[
\sup_{1\leq k\leq m}k^{1/p}e_{k}\left(T\right)\leq c_{p}\sup_{1\leq k\leq m}k^{1/p}a_{k}\left(T\right)\,\,\mbox{for }m=1,2,3\ldots.
\]
\end{prop}

\begin{proof}
As we have shown in Theorem \ref{thm:Normequivalence} we can find
a proper $\varrho$ - norm, which is equivalent to the quasi-norm
of $\mathbb{Y}$. Thus it is enough to show, that the estimate stands
for that case. We proceed by looking at dyadic numbers $n=2^{N},\,N\in\mathbb{N}$.
According to the definition of approximation numbers $a_{n}\left(T\right)$
(Definition \ref{def:Approximationszahlen}), we find operators $A_{j}\in\mathcal{L}\left(\mathbb{X},\mathbb{Y}\right)$
with $\mbox{rank }A_{j}<2^{j}$, such that

\begin{equation}
\Vert T-A_{j}\Vert\leq a_{2^{j}}\left(T\right)+\varepsilon_{j}\,\mbox{for }j=0,1,2,\ldots,N\label{eq:Propositiona}
\end{equation}

and for arbitrary $\varepsilon_{j}>0$. Since we have no equality
in Theorem \ref{thm:Properties a} (R$_{a}$), we are not sure if
it might happen that $a_{2^{j}}\left(T\right)=0$, although $\mbox{rank }T\geq2^{j}$.
If $a_{2^{j}}\left(T\right)\neq0$ we set $\varepsilon_{j}:=a_{2^{j}}\left(T\right)$

\begin{equation}
\Vert T-A_{j}\Vert\leq2a_{2^{j}}\left(T\right)\,\mbox{for }j=0,1,2,\ldots,N\,\label{eq:propositiona_1}
\end{equation}
~

where $A_{0}=0$. For simplicity we may and shall assume that $a_{2^{j}}\neq0$,
thus we can set $\varepsilon_{j}:=a_{2^{j}}\left(T\right)$ for all
$j=0,1,\ldots,N$. The argument is modified in an obvious way otherwise.
We may now take differences $A_{j}-A_{j-1}$ for $j=1,\ldots,N$,
which amounts to $\mbox{rank }\left(A_{j}-A_{j-1}\right)<2^{j+1}$.
With that in mind, we find another representation of $T$

\[
T=\sum_{j=1}^{N}\left(A_{j}-A_{j-1}\right)+\left(T-A_{N}\right).
\]

On the other hand, we may successively conclude that 

\begin{equation}
\left(e_{n_{1}+\ldots+n_{N}-\left(N-1\right)}\left(T\right)\right)^{\varrho}\leq\sum_{j=1}^{N}\left(e_{n_{j}}\left(A_{j}-A_{j-1}\right)\right)^{\varrho}+\Vert T-A_{N}\Vert^{\varrho}\label{eq:Proposition_a2}
\end{equation}

for $n_{j}$ natural numbers to be chosen later.

Without loss of generality, we can say, that $T$ acts between real
quasi-Banach spaces and because of $\mbox{rank }\left(A_{j}-A_{j-1}\right)<2^{j+1}$
and Lemma \ref{lem:Estimation unproved} we have the estimate

\begin{equation}
e_{n_{j}}\left(A_{j}-A_{j-1}\right)^{\varrho}\leq C\cdot2^{-\left(n_{j}-1\right)\varrho/2^{j+1}}\Vert A_{j}-A_{j-1}\Vert^{\varrho},\,\,\mbox{for }j=1,2,\ldots,N\label{eq:Proposition_b}
\end{equation}

for some $C>0$. If $T$ acts between complex quasi-Banach spaces
, we use the second statement of Lemma \ref{lem:Estimation unproved}.
This has no effect on the proof except that we get another constant.
We proceed for the real quasi-Banach space case, by using the triangle
inequality, and have 

\begin{equation}
\Vert A_{j}-A_{j-1}\Vert^{\varrho}\leq\Vert A_{j}-T\Vert^{\varrho}+\Vert T-A_{j-1}\Vert^{\varrho}\overset{\tiny\eqref{eq:propositiona_1}}{\leq}2^{\varrho+1}\left(a_{2^{j-1}}\left(T\right)\right)^{\varrho}\label{eq:Proposition_c}
\end{equation}

for $j=1,2,\ldots,N$, where we also used the monotonicity of the
approximation numbers in the last estimation. If we combine \eqref{eq:Proposition_b}
and \eqref{eq:Proposition_c}, we can conclude that

\begin{equation}
\left(e_{n_{j}}\left(A_{j}-A_{j-1}\right)\right)^{\varrho}\leq C_{2}\cdot2^{-\left(n_{j}-1\right)\varrho/2^{j+1}}\left(a_{2^{j-1}}\left(T\right)\right)^{\varrho},\,\,\mbox{for }j=1,2,\ldots,N\label{eq:Proposition_d}
\end{equation}

for some $C_{2}>0$. Now we estimate \eqref{eq:Proposition_a2} by
using \eqref{eq:propositiona_1} and \eqref{eq:Proposition_d}, to
get

\begin{equation}
\left(e_{n_{1}+\ldots+n_{N}-\left(N-1\right)}\left(T\right)\right)^{\varrho}\leq C_{2}\sum_{j=1}^{N}2^{-\left(n_{j}-1\right)\varrho/2^{j+1}}\left(a_{2^{j-1}}\left(T\right)\right)^{\varrho}+\left(2\cdot a_{2^{N}}\left(T\right)\right)^{\varrho}.\label{eq:Proposition_e}
\end{equation}

We shall now have a look at the sum. It can easily be seen that 

\begin{eqnarray*}
\sum_{j=1}^{N}2^{-\left(n_{j}-1\right)\varrho/2^{j+1}}\left(a_{2^{j-1}}\left(T\right)\right)^{\varrho} & \leq & \left(\sum_{j=1}^{N}2^{-\left(n_{j}-1\right)\frac{\varrho}{2^{j+1}}-\left(j-1\right)\frac{\varrho}{p}}\right)\sup_{1\leq j\leq N}2^{\left(j-1\right)\frac{\varrho}{p}}\left(a_{2^{j-1}}\left(T\right)\right)^{\varrho}\\
 & \leq & \left(\sum_{j=1}^{N}2^{-\left(n_{j}-1\right)\frac{\varrho}{2^{j+1}}-\left(j-1\right)\frac{\varrho}{p}}\right)\sup_{1\leq j\leq2^{N}}j^{\frac{\varrho}{p}}\left(a_{j}\left(T\right)\right)^{\varrho}.
\end{eqnarray*}

Furthermore we can find an upper bound for $\left(a_{2^{N}}\left(T\right)\right)^{\varrho}$,
which is given through $\left(a_{2^{N}}\left(T\right)\right)^{\varrho}\leq2^{-N\varrho/p}\sup_{1\leq j\leq2^{N}}j^{\varrho/p}\left(a_{j}\left(T\right)\right)^{\varrho}$.
Combining these two estimates with \eqref{eq:Proposition_e} amounts
to

\begin{equation}
\left(e_{n_{1}+\ldots+n_{N}-\left(N-1\right)}\left(T\right)\right)^{\varrho}\leq\left(C_{2}\sum_{j=1}^{N}2^{-\left(n_{j}-1\right)\frac{\varrho}{2^{j+1}}-\left(j-1\right)\frac{\varrho}{p}}+2^{\varrho}\cdot2^{-\frac{N\varrho}{p}}\right)\sup_{1\leq j\leq2^{N}}j^{\varrho/p}\left(a_{j}\left(T\right)\right)^{\varrho}.\label{eq:Proposition_f}
\end{equation}

We will now choose the still free natural numbers $n_{j}$ in a way,
that the large sum can be estimated by terms, which can easier be
expressed. Therefore we choose a natural number $1+\frac{1}{p}\leq K\leq2+\frac{1}{p}$
and set $n_{j}=1+K\left(N-j\right)2^{j+1}$ for $j=1,2,\ldots,N$.
We can now see that $2^{-\left(n_{j}-1\right)/2^{j+1}}=2^{-K\left(N-j\right)}$.
Using properties of a finite geometric series, we conclude

\begin{eqnarray*}
\sum_{j=1}^{N}2^{-K\left(N-j\right)\varrho-\left(j-1\right)\varrho/p} & = & 2^{\varrho/p}2^{-\varrho KN}\sum_{j=1}^{N}2^{\left(K-1/p\right)\varrho j}\\
 & = & 2^{\varrho/p}2^{-\varrho KN}2^{\left(K-1/p\right)\varrho}\frac{2^{\left(K-1/p\right)\varrho N}-1}{2^{\left(K-1/p\right)\varrho}-1}\\
 & = & 2^{-KN\varrho}2^{K\varrho}\frac{2^{\left(K-1/p\right)\varrho N}-1}{2^{\left(K-1/p\right)\varrho}-1}\\
 & \leq & C_{3}2^{-N\varrho/p}\,\,\,\mbox{for some }C_{3}>0,
\end{eqnarray*}

where the last inequality is established through $1+\frac{1}{p}\leq K\leq2+\frac{1}{p}$.
Thus we can now estimate \eqref{eq:Proposition_f} through 

\[
\left(e_{n_{1}+\ldots+n_{N}-\left(N-1\right)}\left(T\right)\right)^{\varrho}\leq2^{-N\varrho/p}\cdot C_{4}\sup_{1\leq j\leq2^{N}}j^{\varrho/p}\left(a_{j}\left(T\right)\right)^{\varrho}
\]

for some $C_{4}>0$. We now need to estimate $n_{1}+\ldots+n_{N}-\left(N-1\right)$.
Since we have chosen $n_{j}$ already for $j=1,\ldots,N$, we conclude,
that $n_{1}+\ldots+n_{N}-N=K\cdot\sum_{j=1}^{N}\left(N-j\right)2^{j+1}$.
By induction we derive that $\sum_{j=1}^{N}\left(N-j\right)2^{j+1}=4\left(2^{N}-\left(N+1\right)\right)$
is valid for $N\in\mathbb{N}$ and thus $n_{1}+\ldots+n_{N}-\left(N-1\right)\leq4\cdot K\cdot2^{N}$.
Using the monotonicity of $e_{n}\left(T\right)$ we obtain

\[
\left(e_{4K2^{N}}\left(T\right)\right)^{\varrho}\leq2^{-N\varrho/p}C_{4}\cdot\sup_{1\leq j\leq2^{N}}j^{\varrho/p}\left(a_{j}\left(T\right)\right)^{\varrho},\,\,\mbox{for }N=1,2,\ldots.
\]

We shall next estimate $e_{n}\left(T\right)$ for an arbitrary natural
number $n$. Let us first consider the case, $n\geq8K$ (with $K$
given as above). We then find a natural number $N$ such that $8K2^{N-1}\leq n\leq8K2^{N}$.
Hence 

\[
2^{-N\varrho/p}\sup_{1\leq j\leq2^{N}}j^{\varrho/p}\left(a_{j}\left(T\right)\right)^{\varrho}\leq\left(8K\right)^{\varrho/p}n^{-\varrho/p}\sup_{1\leq j\leq n}j^{\varrho/p}\left(a_{j}\left(T\right)\right)^{\varrho}.
\]

If we again use the monotonicity of the entropy numbers in combination
with $n\geq4K2^{N}$, we have

\[
\left(e_{n}\left(T\right)\right)^{\varrho}\leq\left(e_{4K2^{N}}\left(T\right)\right)^{\varrho}\leq C_{4}\left(8K\right)^{\varrho/p}n^{-\varrho/p}\sup_{1\leq j\leq2^{N}}j^{\varrho/p}\left(a_{j}\left(T\right)\right)^{\varrho}
\]

and this brings us finally to 

\[
n^{\varrho/p}\left(e_{n}\left(T\right)\right)^{\varrho}\leq C_{4}\left(8K\right)^{\varrho/p}\sup_{1\leq j\leq n}j^{\varrho/p}\left(a_{j}\left(T\right)\right)^{\varrho}.
\]

We will now see, that this estimate holds in the case $1\leq n\leq8K$,
since

\[
n^{\varrho/p}\left(e_{n}\left(T\right)\right)^{\varrho}\leq\left(8K\right)^{\varrho/p}\Vert T\Vert^{\varrho}\leq C_{4}\left(8K\right)^{\varrho/p}\sup_{1\leq j\leq n}j^{\varrho/p}\left(a_{j}\left(T\right)\right)^{\varrho}.
\]

We already know by construction that $K\leq2+\frac{1}{p}$ and by
taking the $\varrho$ - th root and the supremum of both sides of
the inequality with respect to $n\leq m$, we finally arrive at 

\[
\sup_{1\leq n\leq m}n^{1/p}e_{n}\left(T\right)\leq c_{p}\sup_{1\leq n\leq m}n^{1/p}a_{n}\left(T\right)\,\,\mbox{for }m=1,2,\ldots.
\]
\end{proof}
This inequality will be needed in the following theorem, which is
again slightly modified taken from \cite[Sect. 3.1, Thm.3.1.1.; Sect. 3.5, Prop. 3.5.3. ]{Carl=000026Stephani},
to fit in the context of quasi-Banach spaces.
\begin{thm}
\label{thm:Relationship of a_n and e_n}(Relationship of $a_{n}$
\& $e_{n}$)

Let $\mathbb{X}$ and $\mathbb{Y}$ be quasi-Banach spaces with constants
$C_{\mathbb{X}}$, $C_{\mathbb{Y}}$ and $T\in\mathcal{L}\left(\mathbb{X},\mathbb{Y}\right)$.
Furthermore let $0<p<\infty,$ then we have
\end{thm}

\begin{lyxlist}{00.00.0000}
\item [{(i)}] $e_{n}\left(T\right)\leq c_{p}\left(\frac{1}{n}\sum_{i=1}^{n}\left(a_{i}\left(T\right)\right)^{p}\right)^{1/p}\,\,\mbox{for }n=1,2,\ldots.$
\item [{(ii)}] $\lim_{n\rightarrow\infty}e_{n}\left(T\right)\leq\lim_{n\rightarrow\infty}a_{n}\left(T\right)$. 
\end{lyxlist}
\begin{proof}
(i)

The proof of the estimate of entropy numbers by the arithmetic mean
of order $p$ of the approximation numbers is an immediate consequence
of Proposition \ref{pro:supremumsdinge}. With the monotonicity of
approximation numbers in mind, we will first have a look at 

\[
k^{1/p}a_{k}\left(T\right)=\left(ka_{k}\left(T\right)^{p}\right)^{1/p}\leq\left(\sum_{i=1}^{k}a_{i}\left(T\right)^{p}\right)^{1/p}.
\]

\begin{equation}
\Longrightarrow\sup_{1\leq k\leq n}k^{1/p}a_{k}\left(T\right)\leq\left(\sum_{i=1}^{n}a_{i}\left(T\right)^{p}\right)^{1/p}\label{eq:mean}
\end{equation}

And now for arbitrary $n\in\mathbb{N}$ we have

\begin{eqnarray*}
e_{n}\left(T\right) & \leq & n^{-1/p}\sup_{1\leq k\leq n}k^{1/p}e_{k}\left(T\right)\leq c_{p}\cdot n^{-1/p}\sup_{1\leq k\leq n}k^{1/p}a_{k}\left(T\right).\\
 & \overset{\tiny\eqref{eq:mean}}{\leq} & c_{p}\left(\frac{1}{n}\sum_{i=1}^{n}a_{i}\left(T\right)^{p}\right)^{1/p}.
\end{eqnarray*}

(ii)

Our first step will be, to show that $\lim_{k\rightarrow\infty}e_{k}\left(T\right)\leq a_{n}\left(T\right)$
for every $n\in\mathbb{N}$. So for fixed $n\in\mathbb{N}$ we may
and shall assume that $a_{n}\left(T\right)\neq0$, since $a_{n}\left(T\right)=0$
would mean, that $\lim_{n\rightarrow\infty}a_{n}\left(T\right)=0$
and therefore $T\in\mathcal{K}\left(\mathbb{X},\mathbb{Y}\right)$
. This however would result in $\lim_{n\rightarrow\infty}e_{n}\left(T\right)=0$,
which is the desired estimate. So let us choose an arbitrary $\delta>a_{n}\left(T\right)$.
By definition of the approximation numbers, we can find an operator
$A\in\mathcal{L}\left(\mathbb{X},\mathbb{Y}\right)$ with $\mbox{rank }A<n$
such that

\[
\Vert T-A\Vert<\delta.
\]

We find that the following inclusion is valid:

\begin{equation}
T\left(\overline{B}_{\mathbb{X}}\right)\subseteq\left(\Vert T-A\Vert^{\varrho}+\Vert A\Vert^{\varrho}\right)^{1/\varrho}\overline{B}_{\mathbb{Y}}.\label{eq:subset}
\end{equation}

Hence if we find a covering of the set $A\left(\overline{B}_{\mathbb{X}}\right)$,
we have also found one for $T\left(\overline{B}_{\mathbb{X}}\right)$.
$A$ is a finite rank operator, because $\mbox{rank }A<n$, hence
it is compact. Therefore for any given $\varepsilon>0$ we can find
finitely many elements $y_{j}\in\mathbb{Y}$, for $j=1,\ldots,m$,
such that 

\[
A\left(\overline{B}_{\mathbb{X}}\right)\subseteq\bigcup_{j=1}^{m}\left\{ y_{j}+\varepsilon\overline{B}_{\mathbb{Y}}\right\} 
\]

indeed is a covering. With \eqref{eq:subset} we can now conclude
that

\[
T\left(\overline{B}\mathbb{_{X}}\right)\subseteq\bigcup_{j=1}^{m}\left\{ y_{j}+\left(\delta^{\varrho}+\varepsilon^{\varrho}\right)^{1/\varrho}\overline{B}_{\mathbb{Y}}\right\} .
\]

By definition and monotonicity of the entropy numbers we get

\[
\lim_{k\rightarrow\infty}e_{k}\left(T\right)\leq e_{n}\left(T\right)\leq\left(\delta^{\varrho}+\varepsilon^{\varrho}\right)^{1/\varrho}.
\]

Since $\varepsilon>0$ was arbitrarily chosen, we let $\varepsilon\downarrow0$
and obtain $\lim_{k\rightarrow\infty}e_{k}\left(T\right)\leq\delta$
and hence by taking the infimum over all such $\delta$, we get $\lim_{k\rightarrow\infty}e_{k}\left(T\right)\leq a_{n}\left(T\right)$
for $n=1,2,\ldots$. Therefore we get

\[
\lim_{k\rightarrow\infty}e_{k}\left(T\right)\leq\lim_{k\rightarrow\infty}a_{k}\left(T\right).
\]
\\
\\
\\
\end{proof}
\begin{rem}
If $\mathbb{H}_{1}$ and $\mathbb{H}_{2}$ are Hilbert spaces and
$T\in\mathcal{L}\left(\mathbb{H}_{1},\mathbb{H}_{2}\right)$, then

\[
\sup_{1\leq k<\infty}2^{-n/k}\left(\prod_{i=1}^{k}a_{i}\left(T\right)\right)^{1/k}\leq e_{n}\left(T\right)\leq14\cdot\sup_{1\leq k<\infty}2^{-n/k}\left(\prod_{i=1}^{k}a_{i}\left(T\right)\right)^{1/k}.
\]

The proof can be found in \cite[Sect. 3.4.; Thm. 3.4.2.]{Carl=000026Stephani}.
\end{rem}

\section{Relationship between Approximation and Kolmogorov Numbers}
\begin{thm}
\label{thm:relationship a d}(Relationship of $a_{n}$ \& $d_{n}$)

Let $\mathbb{X}$ and $\mathbb{Y}$ be quasi-Banach spaces and $T\in\mathcal{L}\left(\mathbb{X},\mathbb{Y}\right)$,
$n\in\mathbb{N}$. Then

\[
d_{n}\left(T\right)\leq a_{n}\left(T\right).
\]
\end{thm}

\begin{proof}
The proof is taken from \cite[Rem. p.50]{Carl=000026Stephani} and
\cite[Lem. 3.34]{HaroskeAT} for the case of Banach spaces and is
slightly altered here again. It starts with an arbitrary $\varepsilon>0$
and $n\in\mathbb{N}$. We know the following

\[
\exists L\in\mathcal{L}\left(\mathbb{X},\mathbb{Y}\right),\,\mbox{rank }L<n\,:\,\Vert T-L\Vert\leq a_{n}\left(T\right)+\varepsilon,
\]

since $a_{n}\left(T\right)$ is given through the best approximating
operator. Now we define $U_{n}:=\mathcal{R}\left(L\right)$ and see
through definition of the norm of an operator 

\[
\exists U_{n},\,\dim U_{n}<n\,\forall x\in\overline{B}_{\mathbb{X}}\,:\,\Vert Tx-Lx\Vert_{\mathbb{Y}}\leq a_{n}\left(T\right)+\varepsilon.
\]

Furthermore we get $Lx=:y\in\mathcal{R}\left(L\right)=U_{n}$, and
therefore

\[
\exists U_{n},\,\dim U_{n}<n\,\forall x\in\overline{B}_{\mathbb{X}}\,\exists y\in U_{n}\,:\,\Vert Tx-y\Vert_{\mathbb{Y}}\leq a_{n}\left(T\right)+\varepsilon.
\]

We have seen, that there exist such $y$ and since the inequality
is valid for arbitrary $x\in\overline{B}_{\mathbb{X}}$, we can take
the supremum of all $x\in\overline{B}_{\mathbb{X}}$ over the infimum
of these $y$, which results to 

\[
\exists U_{n},\,\dim U_{n}<n\,:\,\sup_{\Vert x\Vert_{\mathbb{X}}\leq1}\inf_{y\in U_{n}}\Vert Tx-y\Vert_{\mathbb{Y}}\leq a_{n}\left(T\right)+\varepsilon.
\]

Now taking the infimum over all such subspaces $U_{n}$ with $\dim U_{n}<n$,
yields

\[
d_{n}\left(T\right)<a_{n}\left(T\right)+\varepsilon.
\]

And since $\varepsilon>0$ was arbitrary, we let $\varepsilon\downarrow0$,
which is our wanted result.\\
\\
\end{proof}
\begin{rem}
Under the same conditions as in Theorem \ref{thm:relationship a d},
we find that

\[
a_{n}\left(T\right)\leq\left(2n\right)^{1/2}d_{n}\left(T\right)
\]

is also valid. This inequality is taken from \cite[Sect. 2.4. Prop. 2.4.6.]{Carl=000026Stephani}
in the case of Banach spaces. It is proved there with the help of
the lifting of an operator and their corresponding lifting constants,
which are, in case of Banach spaces, bounded. Since the proof of the
quasi-Banach space case would be completely analogue, we only refer
to the proposition of the above authors.
\end{rem}

~
\begin{rem}
For the matter of completeness, we will add without proof, that in
the case, where $\mathbb{X}$ is a Banach space, $\mathbb{H}$ a Hilbert
space and $T\in\mathcal{L}\left(\mathbb{X},\mathbb{H}\right)$, we
get even an equation

\[
d_{n}\left(T\right)=a_{n}\left(T\right).
\]

A proof can be found in \cite[Sect. 11., Prop. 11.6.2.]{Albrecht Pietsch 2}.
\end{rem}

\section{Relationship between Entropy and Kolmogorov Numbers }

At last we examine entropy numbers and Kolmogorov numbers. Equivalent
to Proposition \ref{pro:supremumsdinge} we can make the following
statement, taken from \cite[Lem 4.4.]{JanVybiral}, which is altered
here only in the exponent.
\begin{prop}
\label{pro:supremumsdinge2}Let $\alpha>0$, $0<p<\infty$ and $\mathbb{X}$
and $\mathbb{Y}$ be two quasi Banach spaces with constants $C_{\mathbb{X}}$
and $C_{\mathbb{Y}}$. Further let $T\in\mathcal{L}\left(\mathbb{X},\mathbb{Y}\right)$.
Then for $n\in\mathbb{N}$ there exists a constant $c>0$ such that 

\[
\sup_{1\leq k\leq n}k^{\alpha}e_{k}\left(T\right)\leq c_{p,\alpha}\cdot\sup_{1\leq k\leq n}k^{\alpha}d_{k}\left(T\right).
\]
\end{prop}

\begin{proof}
According to \cite[Thm. 1.]{Carl} it is enough to show that

\[
\sup_{1\leq k\leq n}k^{\alpha}e_{k}\left(T\right)\leq c_{p,\alpha}\cdot\sup_{1\leq k\leq n}k^{\alpha}s_{k}\left(T\right)
\]

in the case of Banach spaces, where $s_{k}\left(T\right)$ either
denotes Kolmogorov or approximation numbers, since this statement
would be valid for the corresponding other number as well. We have
already shown this relation for approximation numbers in Proposition
\ref{pro:supremumsdinge}, thus the proof is done for Banach spaces.
To extend this result for the desired quasi-Banach spaces, we refer
to \cite[Lem. 1.]{BasBernAna}. 
\end{proof}
\begin{cor}
(Relationship of $e_{n}$ \& $d_{n}$) 

Let $\mathbb{X}$ and $\mathbb{Y}$ be quasi-Banach spaces with constants
$C_{\mathbb{X}}$, $C_{\mathbb{Y}}$ and $T\in\mathcal{L}\left(\mathbb{X},\mathbb{Y}\right)$.
Furthermore let $0<p<\infty,$ then we have

\[
e_{n}\left(T\right)\leq c_{p}\left(\frac{1}{n}\sum_{i=1}^{n}\left(d_{i}\left(T\right)\right)^{p}\right)^{1/p}\,\,\mbox{for }n=1,2,\ldots.
\]
\end{cor}

\begin{proof}
Based on Proposition \ref{pro:supremumsdinge2} with $\alpha=\frac{1}{p}$,
the proof follows Theorem \ref{thm:Relationship of a_n and e_n} (i)
analogously.
\end{proof}

\section{Axiomatic Theory of $s$-Numbers}

If we compare Theorems \ref{thm:(Properties-of-e)}, \ref{thm:Properties a}
and \ref{thm:Properties d} we find many similarities, like the monotonicity
or additivity of the corresponding numbers. There is however another
way to introduce numbers like approximation- and Kolmogorov numbers.
This axiomatic way and goes back to Albrecht Pietsch. For further
detail one may have a look at \cite[Sect. 2.2.]{Albrecht Pietsch}.
To fit in the context of quasi-Banach spaces we will follow the notation
of \cite[Sect. 2.4.]{JanVybiral}.
\begin{defn}
($s$ - numbers )

Let $\mathbb{W}$,$\mathbb{X}$,$\mathbb{Y}$ and $\mathbb{Z}$ be
quasi-Banach spaces and $T\in\mathcal{L}\left(\mathbb{X},\mathbb{Y}\right)$.
A rule $s\,:\,T\rightarrow\left(s_{n}(T)\right)_{n\in\mathbb{N}}$,
which assigns to every operator a scalar sequence is called an $s$
- scale if it satisfies the following conditions
\end{defn}

\begin{lyxlist}{00.00.0000}
\item [{(M$_{s}$)}] $\Vert T\Vert=s_{1}\left(T\right)\geq s_{2}\left(T\right)\geq\ldots\geq0$ 
\item [{(A$_{s}$)}] $s_{m+n-1}\left(T\right)\leq C_{\mathbb{Y}}\left(s_{m}\left(T\right)+s_{n}\left(T\right)\right)$
where $C_{\mathbb{Y}}$ is the constant of the quasi-Banach space
$\mathbb{Y}$ or likewise $s_{m+n-1}\left(T\right)^{\varrho}\leq s_{m}\left(T\right)^{\varrho}+s_{n}\left(T\right)^{\varrho}$,
for an equivalent $\varrho$ - norm with $\varrho\in(0,1]$. $S\in\mathcal{L}\left(\mathbb{X},\mathbb{Y}\right)$
and $n,m\in\mathbb{N}$.
\item [{(S$_{s}$)}] $s_{n}\left(RTU\right)\leq\Vert R\Vert s_{n}\left(T\right)\Vert U\Vert$
for all $U\in\mathcal{L}\left(\mathbb{W},\mathbb{X}\right),$ $T\in\mathcal{L}\left(\mathbb{X},\mathbb{Y}\right)$,
$R\in\mathcal{L}\left(\mathbb{Y},\mathbb{Z}\right)$ and $n\in\mathbb{N}$.
\item [{(R$_{s}$)}] $\mbox{rank }T<n\,\Longrightarrow s_{n}\left(T\right)=0$
\item [{(I$_{s}$)}] $s_{n}\left(\mbox{id : }\ell_{2}^{n}\rightarrow\ell_{2}^{n}\right)=1$
\end{lyxlist}
Furthermore an $s$ - scale is said to be multiplicative if it also
satisfies 
\begin{lyxlist}{00.00.0000}
\item [{(P$_{s}$)}] $s_{m+n-1}\left(RT\right)\leq s_{m}\left(R\right)s_{n}\left(T\right)$
for every $R\in\mathcal{L}\left(\mathbb{Y},\mathbb{Z}\right)$ and
$m,n\in\mathbb{N}$.
\end{lyxlist}
\begin{rem}
At first we see, that entropy numbers do not fit in this construct,
since they do not satisfy (R$_{s}$) as we can see in Lemma \ref{lem:Estimation unproved}. 

As stated above, this is another approach to the theory of approximation-
and Kolmogorov numbers. We can easily make sure, that these numbers
are indeed $s$ - scales, since we have already shown all of the necessary
properties. (Theorem \ref{thm:Properties a} and Theorem \ref{thm:Properties d}) 

Although they will not be a part of this bachelor's thesis we denote
that there are many other concepts of $s$ - scales, as for example 
\end{rem}

\begin{lyxlist}{00.00.0000}
\item [{(i)}] Gelfand numbers : $c_{n}\left(T\right):=\inf\left\{ \Vert TS_{V}^{\mathbb{X}}\Vert\,:\,\mbox{codim }V<n\right\} $,
where $V$ is a subspace of the Banach space $\mathbb{X}$.
\item [{(ii)}] Weyl numbers : $x_{n}\left(T\right):=\sup\left\{ a_{n}\left(TS\right):\,S\in\mathcal{L}\left(\ell_{2},\mathbb{X}\right),\,\Vert S\Vert\leq1\right\} $
\end{lyxlist}
The proof that these numbers are $s$ - scales indeed, can be found
in \cite[Sect. 2.4. Thm. 2.4.3.*, Thm. 2.4.14.*]{Albrecht Pietsch}
and is left out here, since it would go beyond the intended scope
of this bachelor's thesis.

\part{Compact Embeddings}

\section{$\mbox{id : }\ell_{p}^{n}\rightarrow\ell_{q}^{n}$ and e\noun{ntr}opy
numbers}

In this part we will only deal with one specific operator, which is
$T:=\mbox{id }\left(\ell_{p}^{n}\rightarrow\ell_{q}^{n}\right)$ for
given $0<p,q\leq\infty$ and $n\in\mathbb{N}$. For reasons of abbreviation
we denote $e_{k}:=e_{k}\left(T\right)$, $a_{k}:=a_{k}\left(T\right)$
and $d_{k}=d_{k}\left(T\right)$ for the whole following part. But
before we investigate the embeddings, we need to have a look at the
following proposition, taken from \cite[Subsect. 3.2.1., Prop.]{Edmunds and Triebel}.
Since we introduced the sequence spaces $\ell_{p}$ over $\mathbb{R}$
or $\mathbb{C}$, it is not clear whether we deal with real or complex
elements, but $\mathbb{C}^{n}$ may be identified with $\mathbb{R}^{2n}$.
Using this interpretation we make the convention, that with the volume
$\mbox{vol}\overline{B}_{\ell_{p}^{n}}$, we mean the Lebesgue-$2n$-measure
of $\left\{ x\in\mathbb{R}^{2n}:\,\sum_{j=1}^{n}\left(x_{2j-1}^{2}+x_{2j}^{2}\right)^{\frac{p}{2}}\leq1\right\} $
to cover both cases. Our aim here, is to identify $x_{2j-1}$ and
$x_{2j}$ with real and imaginary part of a complex number, but we
reduce it to two real values. 
\begin{prop}
\label{pro:Ballvolume}Let $n\in\mathbb{N}$, then
\end{prop}

\begin{lyxlist}{00.00.0000}
\item [{(i)}] If $0<p\leq\infty$ , then the volume of the unit ball in
$\ell_{p}^{n}$ is $\mbox{vol}\overline{B}_{\ell_{p}^{n}}=\pi^{n}\frac{\Gamma\left(1+\frac{2}{p}\right)^{n}}{\Gamma\left(1+\frac{2n}{p}\right)}$
\item [{(ii)}] There exists a function $\theta\,:\,\left(0,\infty\right)\rightarrow\mathbb{R}$,
with $0<\theta\left(x\right)<\frac{1}{12}$ for all $x>0$, such that
for all $p\in\left(0,\infty\right)$
\item [{~~~}] 
\[
\mbox{vol}\overline{B}_{\ell_{p}^{n}}=2^{n-1}\pi^{\frac{1}{2}\left(3n-1\right)}p^{\frac{-\left(n-1\right)}{2}}n^{-\frac{2n}{p}-\frac{1}{2}}\exp\left(n\theta\left(2/p\right)p/2-\theta\left(2n/p\right)p/2n\right).
\]
\end{lyxlist}
\begin{proof}
Since this statement is not crucial for the topic of this bachelor's
thesis and will be needed later on only for a matter of constants,
it will be used without proof. Nevertheless the proof can be found
in \cite[Subsect. 3.2.1., Prop.]{Edmunds and Triebel}.
\end{proof}
The following proposition is taken from \cite[Subsect. 3.2.2., Prop.]{Edmunds and Triebel}.
\begin{prop}
\label{pro:Edmunds Triebel Proposition}($e_{k}$ of $\mbox{id : }\ell_{p}^{n}\longrightarrow\ell_{q}^{n}$,
upper estimate)

Let $0<p\leq q\leq\infty$ then

\begin{equation}
e_{k}\leq c_{p,q}\cdot\begin{cases}
1 & \mbox{if }1\leq k\leq\log_{2}\left(2n\right)\\
\left(k^{-1}\log_{2}\left(1+\frac{2n}{k}\right)\right)^{\frac{1}{p}-\frac{1}{q}} & \mbox{if }\log_{2}\left(2n\right)\leq k\leq2n\\
2^{-\frac{k}{2n}}\left(2n\right)^{\frac{1}{q}-\frac{1}{p}} & \mbox{if }k\geq2n
\end{cases}\label{eq:TheoremEmbeddingE}
\end{equation}

where $c_{p,q}>0$ is a constant independent of $n$ and $k$.
\end{prop}

\begin{proof}
First of all, for a matter of abbreviation we introduce the notation
$\overline{B}_{p}^{n}:=\overline{B}_{\ell_{p}^{n}}$. To proof the
whole estimate, we are going to need four steps. We begin with the
\emph{first step}, which deals with large $k\geq2n$ and $0<p\leq q\leq1.$
We set $r=2^{-\frac{k}{2n}}\left(2n\right)^{\frac{1}{q}-\frac{1}{p}}$
and furthermore $K=K(r)$ be the maximal number of points $y^{j}\in\overline{B}_{p}^{n}$
with $\Vert y^{j}-y^{m}\Vert_{q}>r$ if $j\neq m$. Now for given
$z\in\overline{B}_{q}^{n}$ and by choice of $r$, we can show that
for $p\leq1$

\begin{eqnarray}
\Vert y^{j}+rz\Vert_{\ell_{p}^{n}}^{p} & \leq & 1+r^{p}\Vert z\Vert_{\ell_{p}^{n}}^{p}\nonumber \\
 & \leq & 1+r^{p}\Vert z\Vert_{\ell_{q}^{n}}^{p}n^{p\left(\frac{1}{p}-\frac{1}{q}\right)}\leq2,\label{eq:balls}
\end{eqnarray}

where the second estimate is obtained through Hölder's inequality. 

Now let $\left\{ y^{j}:\,j=1,\ldots,K\right\} $ be such a set, which
is maximal in the above sense. Then clearly we get

\begin{equation}
\overline{B}_{p}^{n}\subset\bigcup_{j=1}^{K}\left\{ y^{j}+r\overline{B}_{q}^{n}\right\} \overset{\eqref{eq:balls}}{\subset}2^{\frac{1}{p}}\overline{B}_{p}^{n}.\label{eq:Ballinclusion}
\end{equation}

If we have a look at the balls $y^{j}+2^{-\frac{1}{q}}r\overline{B}_{q}^{n}$
for $j=1,\ldots,K$ and assume, that there exists an element $z$
which is in two of these balls, then 

\begin{equation}
\Vert y^{j}-y^{m}\Vert_{\ell_{q}^{n}}^{q}\leq\underset{\tiny\leq2^{-1}r^{q}}{\underbrace{\Vert y^{j}-z\Vert_{\ell_{q}^{n}}^{q}}}+\underset{\tiny\leq2^{-1}r^{q}}{\underbrace{\Vert y^{m}-z\Vert_{\ell_{q}^{n}}^{q}}}\leq r^{q},\,\,\mbox{for }q\leq1,\label{eq:triangle2}
\end{equation}

which is by choice of the $y^{j}$ only possible, if $m=j$, hence
the balls are disjoint. Together with \eqref{eq:Ballinclusion} this
yields

\[
K2^{\frac{-2n}{q}}r^{2n}\mbox{vol}\overline{B}_{q}^{n}\leq2^{\frac{2n}{p}}\mbox{vol}\overline{B}_{p}^{n},
\]

where the constants arise, because of our convention of identifying
$\mathbb{C}^{n}$ with $\mathbb{R}^{2n}$. With Proposition \ref{pro:Ballvolume}
(ii) we see that for some positive constant $c=c(p,q)>0,$ independent
of $k,n$ 

\[
\mbox{vol}\overline{B}_{p}^{n}\leq c^{2n}\left(2n\right)^{-2n\left(\frac{1}{p}-\frac{1}{q}\right)}\mbox{vol}\overline{B}_{q}^{n}
\]

where we again used Hölder's inequality. Combining this estimate with
the preceding one and our choice of $r$ yields 

\[
K\leq2^{k+cn}
\]

and hence

\begin{equation}
e_{k+cn}\leq r=2^{\frac{-k}{2n}}\left(2n\right)^{\frac{1}{q}-\frac{1}{p}}\,\,\mbox{if }k\geq2n,\label{eq:step1}
\end{equation}

by definition of $e_{k}$. Since $k+cn$ does not necessarily has
to be a natural number, we will from now on use the notation $e_{\lambda}=e_{\left\lfloor \lambda\right\rfloor +1}$
if $\lambda\geq1$, where $\left\lfloor \lambda\right\rfloor $ denotes
the smallest integer bigger than $\lambda$. 

\[
e_{\underset{\tilde{k}}{\underbrace{k+cn}}}=e_{\tilde{k}}\leq2^{\frac{-(\tilde{k}-cn)}{2n}}\left(2n\right)^{\frac{1}{q}-\frac{1}{p}}=2^{\frac{c}{2}}\cdot2^{\frac{-\tilde{k}}{2n}}\left(2n\right)^{\frac{1}{q}-\frac{1}{p}}
\]

Thus if $k\geq c_{1}n$ for $c_{1}>1$ independent of $n$ and $k$,
we have proved \eqref{eq:TheoremEmbeddingE} for $0<p\leq q\leq1$.

Our \emph{second step} will be only a modification of the first one.
We still assume that $k\geq2n$ is large and notice, that the argument
above holds for all $0<p\leq q\leq\infty$, since we only used properties
of the $p$ - and accordingly the $q$-norms, thus we only need to
alter these points by using the triangle inequality in \eqref{eq:balls}
and \eqref{eq:triangle2}. The rest of the proof proceeds analogously.
In particular, we consider the case $0<p=q\leq\infty$ . We know by
Theorem \ref{thm:(Properties-of-e)}, that $e_{k}\left(T\right)\leq\Vert T\Vert$
and since 

\[
\Vert\mbox{id : }\ell_{p}^{n}\rightarrow\ell_{q}^{n}\Vert=\sup_{\Vert x\Vert_{\ell_{p}^{n}}\leq1}\Vert x\Vert_{\ell_{q}^{n}}\leq\sup_{\Vert x\Vert_{\ell_{p}^{n}}\leq1}\underset{\tiny=1}{\underbrace{n^{p\left(\frac{1}{p}-\frac{1}{q}\right)}}}\Vert x\Vert_{\ell_{p}^{n}}\leq1,
\]

hence $e_{k}\leq1$ for all $k\in\mathbb{N}$. This proves the statement
for $p=q$ for mid-ranged and small $k$.

We proceed with the \emph{third step}, which covers the case $0<p<q=\infty$
and $1\leq k\leq c_{1}n$. Here $c_{1}$ has the same meaning as before.
We choose a second constant $c_{2}>\left(\frac{1}{c_{1}}\log_{2}\left(1+\frac{1}{c_{1}}\right)\right)^{-1/p}$
and set

\begin{equation}
\sigma:=c_{2}\left(k^{-1}\log_{2}\left(\frac{n}{k}+1\right)\right)^{1/p}=c_{2}n^{-1/p}\left(\frac{n}{k}\log_{2}\left(\frac{n}{k}+1\right)\right)^{1/p}>n^{-\frac{1}{p}},\label{eq:embeddingestimation}
\end{equation}

where the last estimation occurs by choice of $c_{2}$. We define
$n_{\sigma}$ as the maximal number of components $y_{n}$, which
a point $y=\left(y_{1},y_{2},\ldots,y_{n}\right)\in\overline{B}_{p}^{n}$
may have, for which $\vert y_{n}\vert>\sigma$. By the preceding estimate
\eqref{eq:embeddingestimation}, we have $n_{\sigma}<n$, otherwise
$y\notin\overline{B}_{p}^{n}$. Furthermore we get $n_{\sigma}\sigma^{p}\leq1$,
hence $n_{\sigma}\leq\sigma^{-p}$. Let us now assume that $\sigma^{-p}\in\mathbb{N}$
and $n_{\sigma}\sigma^{p}=1$. This is possible, because we can always
find such a number $\sigma$ which satisfies the above conditions.
We set

\[
e_{k}^{\left(\sigma\right)}:=e_{k}\left(\mbox{id}:\ell_{p}^{n_{\sigma}}\rightarrow\ell_{\infty}^{n_{\sigma}}\right)
\]

and with \eqref{eq:step1} from the first step, where we set $k=c_{1}\sigma^{-p}$,
we know that 

\begin{equation}
e_{c_{1}\sigma^{-p}}^{\left(\sigma\right)}\leq c_{3}n_{\sigma}^{-1/p}=c_{3}\sigma,\label{eq:entropyestimation}
\end{equation}

for $c_{1}\geq1$ and $c_{3}\geq1$. This estimate means, that we
need $2^{c_{1}\sigma^{-p}}$ balls in $\ell_{\infty}^{n}$ with radius
$c_{3}\sigma$ to cover $\overline{B}_{p}^{n_{\sigma}}$. But since
we want to cover the whole $\overline{B}_{p}^{n}$, we need to know
in how many ways we can select $n_{\sigma}$ coordinates out of $n$.
The number of possibilities is given through ${n \choose n_{\sigma}}$
and therefore we know, that we need $2^{c_{1}\sigma^{-p}}{n \choose n_{\sigma}}$
balls in $\ell_{\infty}^{n}$ with radius $c_{3}\sigma$ to cover
$\overline{B}_{p}^{n}$. To estimate further, we need the fact, that
for given natural numbers $N,K\in\mathbb{N}_{0}$, with $N\geq K$
we have an upper bound for the binomial coefficient through ${N \choose K}\leq\frac{N^{K}}{K!}$.
By using this and properties of the logarithm, we get

\begin{eqnarray*}
\log_{2}{n \choose n_{\sigma}} & \leq & \log_{2}\frac{n^{n_{\sigma}}}{n_{\sigma}!}=n_{\sigma}\log_{2}n-\sum_{j=1}^{n_{\sigma}}\log_{2}j\\
 & \leq & n_{\sigma}\log_{2}n-n_{\sigma}\log_{2}n_{\sigma}+cn_{\sigma}\\
 & \leq & c'n_{\sigma}\log_{2}\left(\frac{n}{n_{\sigma}}+1\right)
\end{eqnarray*}

where $c$ and $c'$ denote some positive constants. Combining this
estimate, with the above construction, we know, that we need 

\begin{equation}
2^{c_{4}\sigma^{-p}\log_{2}\left(\frac{n}{n_{\sigma}}+1\right)}=2^{c_{4}\sigma^{-p}\log_{2}\left(n\sigma^{p}+1\right)}\label{eq:Balls}
\end{equation}

balls in $\ell_{\infty}^{n}$ with radius $c_{3}\sigma$ to cover
$\overline{B}_{p}^{n}$, where $c_{4}>0$ is independent of $k$ and
$n$. Furthermore with \eqref{eq:embeddingestimation} and the logarithm
properties, we get 

\begin{eqnarray*}
\log_{2}\left(n\sigma^{p}+1\right) & = & \log_{2}\left(c_{2}^{p}\frac{n}{k}\log_{2}\left(\frac{n}{k}+1\right)+1\right)\\
 & \leq & c''\cdot\log_{2}\left(\frac{n}{k}+1\right)\left[\frac{\log_{2}c_{2}^{p}+\log_{2}\left(\frac{n}{k}+1\right)+\log_{2}\log_{2}\left(\frac{n}{k}+1\right)}{\log_{2}\left(\frac{n}{k}+1\right)}\right]\\
 & =c''\cdot & \log_{2}\left(\frac{n}{k}+1\right)\left[1+\underset{\leq c'''}{\underbrace{\frac{\log_{2}c_{2}^{p}+\log_{2}\log_{2}\left(\frac{n}{k}+1\right)}{\log_{2}\left(\frac{n}{k}+1\right)}}}\right]\\
 & \leq & \tilde{c}\cdot\log_{2}\left(\frac{n}{k}+1\right)
\end{eqnarray*}

for some constants $c''$, $c'''$ and $\tilde{c}$. Hence we can
estimate \eqref{eq:Balls} from above with $2^{c_{5}k}$ for a positive
constant $c_{5}\geq1$ independent of $n$ and $k$ by using \eqref{eq:embeddingestimation}
. Hence with \eqref{eq:entropyestimation} we finally get 

\begin{equation}
e_{c_{5}k}\leq c_{6}\frac{\sigma}{c_{2}}=c_{6}\left(\frac{1}{k}\log_{2}\left(\frac{n}{k}+1\right)\right)^{\frac{1}{p}}\,\,\,\mbox{if }1\leq k\leq c_{1}n,\label{eq:estimationxyz}
\end{equation}

where $c_{6}>0$ is also independent of $n$ and $k$. \eqref{eq:TheoremEmbeddingE}
follows, if we have in mind that $e_{k}\leq1$ for all $k\in\mathbb{N}$
always assuming that $0<p<\infty$ and $q=\infty$.

We now advance to the \emph{fourth step}, which deals with $0<p<q<\infty$
and $1\leq k\leq c_{1}n$, where $c_{1}$ has still the same meaning
as in the first step. We show \eqref{eq:TheoremEmbeddingE} by using
\cite[Subsect. 1.3.2., Thm. 1.]{Edmunds and Triebel}, with the following
cast and $\theta\in\left(0,1\right)$

\[
A=B_{0}=\ell_{p}^{n},\,B_{1}=\ell_{\infty}^{n}\,,\,B_{\theta}=\ell_{q}^{n},\,\,\mbox{where }\frac{1}{q}=\frac{\left(1-\theta\right)}{p}.
\]

Now if we consider the premises of the theorem, we have to check that
$\ell_{p}^{n}\cap\ell_{\infty}^{n}\subset\ell_{q}^{n}\subset\ell_{p}^{n}+\ell_{\infty}^{n}$
, which can be rewritten as $\ell_{p}^{n}\subset\ell_{q}^{n}\subset\ell_{\infty}^{n}$
since we have $\ell_{p}^{n}\cap\ell_{\infty}^{n}=\ell_{p}^{n}$ and
$\ell_{p}^{n}+\ell_{\infty}^{n}=\ell_{\infty}^{n}$. The statement
follows from the monotone alignment of the sequence spaces. 

We may now conclude, that

\begin{eqnarray*}
e_{k+1-1}\left(\mbox{id : }\ell_{p}^{n}\rightarrow\ell_{q}^{n}\right) & \leq & 2^{\frac{1}{p}}\underset{\leq1}{\underbrace{e_{1}^{1-\theta}\left(\mbox{id : }\ell_{p}^{n}\rightarrow\ell_{p}^{n}\right)}}e_{k}^{\theta}\left(\mbox{id : }\ell_{p}^{n}\rightarrow\ell_{\infty}^{n}\right)\\
 & \leq & 2^{\frac{1}{p}}e_{k}^{\theta}\left(\mbox{id : }\ell_{p}^{n}\rightarrow\ell_{\infty}^{n}\right).
\end{eqnarray*}

This yields $e_{k}\left(\mbox{id :\,}\ell_{p}^{n}\rightarrow\ell_{q}^{n}\right)\leq ce_{k}^{\theta}\left(\mbox{id }:\,\ell_{p}^{n}\rightarrow\ell_{\infty}^{n}\right)$
and we use \eqref{eq:estimationxyz} with $\frac{\theta}{p}=\frac{1}{p}-\frac{1}{q}$
. 

This, and the fact that $e_{k}\leq1$ $\forall k\in\mathbb{N}$ proves
the theorem.
\end{proof}
We have established an upper estimate for the $k$ - th entropy number
of the identity operator between the two spaces $\ell_{p}^{n}$ and
$\ell_{q}^{n}$. As we already know, this operator is compact, since
it maps between two finite dimensional quasi-Banach spaces. This can
also be seen if $k\rightarrow\infty$, as we have shown in the properties
of entropy numbers. We will now go on by giving a lower estimate for
both large and small $k$ by the following theorem, taken from \cite[Sect. 7, Prop. 7.2, Thm. 7.3]{Triebel1997}.
We will deal with mid-ranged $k$ later on.
\begin{prop}
\label{pro:Triebel Proposition}($e_{k}$ of $\mbox{id : }\ell_{p}^{n}\longrightarrow\ell_{q}^{n}$,
lower estimate I)

Let $0<p\leq q\leq\infty$ and $k\in\mathbb{N}$ then

\[
e_{k}\geq c\cdot\begin{cases}
1 & \mbox{if }1\leq k\leq\log_{2}\left(2n\right)\\
2^{-\frac{k}{2n}}\left(2n\right)^{\frac{1}{q}-\frac{1}{p}} & \mbox{if }k\in\mathbb{N}
\end{cases}
\]

for some positive constant $c$, which is independent of $k$ and
$n$ but may depend on $p$ and $q$.
\end{prop}

\begin{proof}
In the \emph{first step} of this proof, we will show the first inequality.
Let $y\in\ell_{p}^{n}$ where all components are zero except for one,
which is either $1$ or $-1$. Then we know, that there exist $2n$
such elements in $\ell_{p}^{n}$, which also happen to belong to $\overline{B}_{\ell_{p}^{n}}$
and $\overline{B}_{\ell_{q}^{n}}$. Let us now assume that $y^{1}$
and $y^{2}$ are two such elements belonging to the same $\varepsilon$-
ball in $\ell_{q}^{n}$. That means 

\[
y^{1}\in\left\{ x+\varepsilon\overline{B}_{\ell_{q}^{n}}\right\} \,\,\,\mbox{and}\,\,\,y^{2}\in\left\{ x+\varepsilon\overline{B}_{\ell_{q}^{n}}\right\} \,\,\,\,\,\,\mbox{for some}\,x\in\ell_{q}^{n}.
\]

Since we want to cover the cases where $0<q<1$ and $1\leq q\leq\infty$,
let $\overline{q}=\min\left\{ 1,q\right\} $. Next we consider a constant
$c>0$, which is independent of $n$ and $q$, and satisfies

\[
c\leq\Vert y^{1}-y^{2}\Vert_{\ell_{q}^{n}}^{\overline{q}}\leq\Vert y^{1}-x\Vert_{\ell_{q}^{n}}^{\overline{q}}+\Vert x-y^{2}\Vert_{\ell_{q}^{n}}^{\overline{q}}\leq2\varepsilon^{\overline{q}}.
\]

The wanted estimate follows by definition of the entropy numbers (since
we unify all these $\varepsilon$- balls containing $2$ elements
and take the infimum over all such $\varepsilon$) the preceding estimate
and the fact that $k\leq\log_{2}2n$ implies $2^{k-1}<2n$. 

We proceed with our \emph{second step}, which deals with the second
inequality. We choose $\varepsilon>0$ such, that $\overline{B}_{\ell_{p}^{n}}$
is covered by $2^{k-1}$ balls in $\ell_{q}^{n}$ with radius $\varepsilon$.
With the convention that $\mathbb{C}^{n}$ is identified with $\mathbb{R}^{2n}$,
this yields for appropriate $\varepsilon$

\begin{eqnarray}
\mbox{vol}\overline{B}_{\ell_{p}^{n}} & \leq & 2^{k-1}\varepsilon^{2n}\mbox{vol}\overline{B}_{\ell_{q}^{n}}\nonumber \\
 & \leq & 2^{k}e_{k}^{2n}\mbox{vol}\overline{B}_{\ell_{q}^{n}}.\label{eq:ballvolumeestimation}
\end{eqnarray}

Now according to \cite[Sect. 7, 7.1]{Triebel1997} for $0<p\leq\infty$
there exist two positive constants $c_{1},c_{2}$ such that

\[
c_{1}n^{-\frac{1}{p}}\leq\left(\mbox{vol}\overline{B}_{\ell_{p}^{n}}\right)^{\frac{1}{2n}}\leq c_{2}n^{-\frac{1}{p}}.
\]

If we combine this with \eqref{eq:ballvolumeestimation}, we obtain
the desired inequality.
\end{proof}
Since there is only one estimate missing for the case of mid-ranged
$k$, we will now have a look at this case. Therefore we follow the
results of a paper by \cite{JournalKuehn}, which deals exactly with
this missing case. The results can almost directly be transferred,
except for the fact, that the sequence spaces $\ell_{p}$ are complex
in this bachelor's thesis.
\begin{lem}
\label{lem:K=0000FCehn Lemma}($e_{k}$ of $\mbox{id : }\ell_{p}^{n}\longrightarrow\ell_{q}^{n}$,
lower estimate II)

Let $0<p\leq q\leq\infty$. Then

\[
e_{k}\geq c\cdot\left(k^{-1}\log_{2}\left(1+\frac{2n}{k}\right)\right)^{\frac{1}{p}-\frac{1}{q}}
\]

for some positive constant $c$ which is independent of $n$ and $k$
but may depend on $p$ and $q$.
\end{lem}

\begin{proof}
In this proof, we will first consider the spaces $\ell_{p}^{n}$ and
$\ell_{q}^{n}$ as real valued and begin with two arbitrary integers
$n,m\in\mathbb{N}$ with $n\geq4$ and $1\leq m\leq\frac{n}{4}$ and
define the set

\[
S:=\left\{ x=\left(x_{j}\right)_{j=1}^{n}\in\left\{ -1,0,1\right\} ^{n}:\,\sum_{j=1}^{n}\vert x_{j}\vert=2m\right\} .
\]

It is plain to see, that $\#S={n \choose 2m}\cdot2^{2m}$ , since
we need exactly $2m$ components of every element, which are not zero.
There are ${n \choose 2m}$ ways to choose from them and because we
can only choose between $-1$ and $1$, there are $2^{2m}$ ways to
design such an element. Furthermore we notice, that $\left(2m\right)^{-1/p}S$
is contained in the unit sphere of $\ell_{p}^{n}$. Let $h$ be the
Hamming distance on $S$, which is

\[
h\left(x,y\right):=\#\left\{ j\in\left\{ 1,\ldots,n\right\} :x_{j}\neq y_{j}\right\} .
\]

We observe, that for fixed $x\in S$ we get an upper bound for 

\[
\#\left\{ y\in S:\,h\left(x,y\right)\leq m\right\} \leq{n \choose m}\cdot3^{m},
\]

since we can obtain every element $y\in S$ with $h\left(x,y\right)\leq m$
as follows: If we choose an arbitrary set $J\subset\left\{ 1,\ldots,n\right\} $
with $\#J=m$, then we set $y_{j}=x_{j}$ for $j\notin J$ and choose
$y_{j}\in\left\{ -1,0,1\right\} $ arbitrarily for $j\in J$. We proceed
by defining an arbitrary subset $A\subset S$ with a cardinality not
exceeding $a:={n \choose 2m}/{n \choose m}$. Hence

\begin{eqnarray*}
\#\left\{ y\in S:\,\exists x\in A\,\mbox{with }h\left(x,y\right)\leq m\right\}  & \leq & \#A\cdot{n \choose m}\cdot3^{m}\\
 & \leq & {n \choose 2m}\cdot3^{m}<\#S.
\end{eqnarray*}

Through this estimate, it has been shown, that we can find an element
$y\in S$ with $h\left(x,y\right)>m$ for all $x\in A$. Therefore
we may \emph{inductively} construct a subset $\tilde{A}\subseteq S$
with $\#\tilde{A}>a$ and $h\left(x,y\right)>m$ for $x,y\in\tilde{A},\,x\neq y$.
For such $x,y$ we conclude $\Vert x-y\Vert_{q}>m^{1/q}$ . Now we
see that $\left(2m\right)^{-1/p}\tilde{A}\subset\overline{B}_{\ell_{p}^{n}}$.
As we have already established, this set has a cardinality larger
than $a$. Furthermore we see, that the elements of this set have
a distance of $\Vert x-y\Vert_{q}>\left(2m\right)^{-1/p}\cdot m^{1/q}=:\varepsilon$.
If we now set $k:=\log_{2}a$ and use the notation of entropy numbers
$e_{\lambda}$ with $\lambda\geq1$ of the first step of Proposition
\ref{pro:Edmunds Triebel Proposition} we see 

\begin{equation}
e_{k}\geq\frac{\varepsilon}{2}=c_{1}m^{\frac{1}{q}-\frac{1}{p}},\label{eq:kpehnestimation}
\end{equation}

where $c_{1}>0$ is independent of $k$ or $n$. Now we have a closer
look at $a$, which is

\[
a=\frac{{n \choose 2m}}{{n \choose m}}=\frac{m!\left(n-m\right)!}{\left(2m\right)!\left(n-2m\right)!}=\prod_{j=1}^{m}\frac{n-2m+j}{m+j}.
\]

By our choice of $n$ and $m$ we notice that $f\left(x\right)=\frac{n-2m+x}{m+x}$
decreases for $x>0$. Therefore we can estimate $a$ through $\left(\frac{n-m}{2m}\right)^{m}\leq a\leq\left(\frac{n-2m}{m}\right)^{m}$
and get

\begin{equation}
c_{2}m\log_{2}\left(\frac{n}{m}\right)\leq m\log_{2}\left(\frac{n-m}{2m}\right)\leq k\leq m\log_{2}\left(\frac{n-2m}{m}\right)\leq m\log_{2}\left(\frac{n}{m}\right)\label{eq:k=0000FChnestimation2}
\end{equation}

for some $c_{2}>0$, which is independent of $n$ and $m$. Furthermore
we see that the function $g\left(x\right)=x\cdot\log_{2}\left(\frac{n}{x}\right)$
is strictly increasing on $\left[1;\frac{n}{4}\right]$ and maps this
interval on $\left[\log_{2}n;\frac{n}{2}\right]$. Since this function
is strictly increasing and continuous, its inverse exists on the latter
interval and we see that $x\leq\frac{y}{\log_{2}\left(\frac{n}{y}\right)}$,
by

\begin{eqnarray*}
y & = & x\cdot\log_{2}\left(\frac{n}{x}\right)\Longleftrightarrow\frac{n}{y}=\frac{n}{x}\cdot\frac{1}{\log_{2}\left(\frac{n}{x}\right)}\\
\Longrightarrow\log_{2}\left(\frac{n}{y}\right) & = & \log_{2}\left(\frac{n}{x}\right)-\log_{2}\log_{2}\left(\frac{n}{x}\right)\\
\Longrightarrow\frac{y}{\log_{2}\left(\frac{n}{y}\right)} & = & x\cdot\frac{\log_{2}\left(\frac{n}{x}\right)}{\log_{2}\left(\frac{n}{x}\right)\left[1-\frac{\log_{2}\log_{2}\left(\frac{n}{x}\right)}{\log_{2}\left(\frac{n}{x}\right)}\right]}\geq x.
\end{eqnarray*}

The last inequality occurs because $x\in\left[1;\frac{n}{4}\right]$.
We now consider $\log_{2}n\leq k\leq\frac{c_{2}n}{2}$ and set $x=m$
as well as $y=k$ and get $m\leq2\cdot\frac{k}{\log_{2}\left(\frac{n}{k}+1\right)}$,
because $\frac{n}{k}\geq2$ and therefore $2\cdot\log_{2}\left(\frac{n}{k}\right)\geq\log_{2}\left(\frac{n}{k}+1\right)$.

Furthermore we conclude with \eqref{eq:kpehnestimation} and \eqref{eq:k=0000FChnestimation2}

\begin{eqnarray*}
e_{k} & \geq & c\cdot\left(\frac{\log_{2}\left(1+\frac{n}{k}\right)}{k}\cdot\underset{\geq c_{3}}{\underbrace{\frac{k}{m\log_{2}\left(\frac{n}{m}+1\right)}}}\cdot\underset{\geq c_{4}}{\underbrace{\frac{\log_{2}\left(\frac{n}{m}+1\right)}{\log_{2}\left(\frac{n}{k}+1\right)}}}\right)^{1/p-1/q}\\
 & \geq & c'\left(\frac{\log_{2}\left(1+\frac{n}{k}\right)}{k}\right)^{1/p-1/q}\,\,\,\mbox{for }\,\,\,\log_{2}n\leq k\leq\frac{c_{2}n}{2}
\end{eqnarray*}

for $c'>0$, which is independent of $n$ and $k$. The lower estimate
$c_{4}$ arises from the following

\[
\frac{\log_{2}\left(\frac{n}{m}+1\right)}{\log_{2}\left(\frac{n}{k}+1\right)}\geq\frac{\log_{2}n}{\log_{2}\left(\frac{n}{\log_{2}n}+1\right)}\geq\tilde{c}\cdot\frac{\log_{2}n}{\log_{2}\left(\frac{n}{\log_{2}n}\right)}=\tilde{c}\cdot\frac{\log_{2}n}{\log_{2}n-\log_{2}\log_{2}n}\geq c_{4}.
\]

The case $\frac{c_{2}n}{2}\leq k\leq n$ follows from the monotonicity
of entropy numbers and a lower estimate, which we get from \eqref{eq:kpehnestimation}
through 

\[
e_{n}\geq c''n^{\frac{1}{q}-\frac{1}{p}},
\]

for a certain $c''>0$ (also independent of $n$ and $k$), since
for these 

\[
e_{k}\geq c'''\cdot\left(\frac{\log_{2}\left(1+\frac{n}{k}\right)}{k}\right)^{1/p-1/q}
\]

is valid. 

Of course this was only the proof for real-valued sequence spaces
$\ell_{p}^{n}$ but the proof is analogously, if we set $n=2l$ and
consider complex-valued sequence spaces.
\end{proof}
Now we estimated every case and finish this part by summarizing the
results in the following theorem.
\begin{thm}
\label{thm:(Behaviour-of-e)}(Behavior of $e_{k}\left(\mbox{id : }\ell_{p}^{n}\rightarrow\ell_{q}^{n}\right)$)

Let $0<p\leq q\leq\infty$. Then

\[
e_{k}\left(\mbox{id : }\ell_{p}^{n}\rightarrow\ell_{q}^{n}\right)\sim\begin{cases}
1 & \mbox{if }1\leq k\leq\log_{2}\left(2n\right)\\
\left(k^{-1}\log_{2}\left(1+\frac{2n}{k}\right)\right)^{\frac{1}{p}-\frac{1}{q}} & \mbox{if }\log_{2}\left(2n\right)\leq k\leq2n\\
2^{-\frac{k}{2n}}\left(2n\right)^{\frac{1}{q}-\frac{1}{p}} & \mbox{if }k\geq2n.
\end{cases}
\]
\end{thm}

\begin{proof}
The proof rests on Propositions \ref{pro:Edmunds Triebel Proposition},
\ref{pro:Triebel Proposition} and on Lemma \ref{lem:K=0000FCehn Lemma}. 
\end{proof}

\section{$\mbox{id : }\ell_{p}^{n}\rightarrow\ell_{q}^{n}$ and approximation
numbers }

Next we will give upper and lower estimates of the numbers $a_{k}\left(\mbox{id : }\ell_{p}^{n}\rightarrow\ell_{q}^{n}\right)$.
Here, we will mostly follow \cite[Subsect. 3.2.3.]{Edmunds and Triebel}
and as these authors have done, we will denote real valued sequence
spaces by $\ell_{p}^{n,\mathbb{R}}$ and correspondingly $a_{k}^{\mathbb{R}}:=\left(\mbox{id : }\ell_{p}^{n,\mathbb{R}}\rightarrow\ell_{q}^{n,\mathbb{R}}\right)$
the approximation numbers of the identity operator, acting between
real valued sequence spaces. We will first mention an estimate for
the latter ones and proceed by giving a relation between approximation
numbers of the identity operator acting between real and complex valued
sequence spaces later on.
\begin{thm}
\label{thm:real-valued-approximation-numbers}Let $k\leq n$. We define

\[
\Phi\left(n,k,p,q\right)=\begin{cases}
\left(\min\left\{ 1,n^{\frac{1}{q}}k^{-\frac{1}{2}}\right\} \right)^{\frac{\frac{1}{p}-\frac{1}{q}}{\frac{1}{2}-\frac{1}{q}}} & \mbox{if }2\leq p<q\leq\infty\\
\max\left\{ n^{\frac{1}{q}-\frac{1}{p}},\min\left\{ 1,n^{\frac{1}{q}}k^{-\frac{1}{2}}\right\} \sqrt{1-\frac{k}{n}}\right\}  & \mbox{if }1\leq p<2\leq q\leq\infty\\
\max\left\{ n^{\frac{1}{q}-\frac{1}{p}},\left(\sqrt{1-\frac{k}{n}}\right)^{\frac{\frac{1}{p}-\frac{1}{q}}{\frac{1}{p}-\frac{1}{2}}}\right\}  & \mbox{if }1\leq p<q\leq2
\end{cases}
\]

and 

\[
\Psi\left(n,k,p,q\right)=\begin{cases}
\Phi\left(n,k,p,q\right) & \mbox{if }1\leq p<q<p'\\
\Phi\left(n,k,q',p'\right) & \mbox{if }\mbox{\ensuremath{\max}}\left\{ p,p'\right\} <q\leq\infty
\end{cases}
\]

where for given $1\leq p,q\leq\infty$, $p'$ and $q'$ are given
through $\frac{1}{p}+\frac{1}{p'}=1$ and $\frac{1}{q}+\frac{1}{q'}=1$
respectively. 

(i) If we assume that $1\leq p<q\leq\infty$ and $\left(p,q\right)\neq\left(1,\infty\right)$,
Then
\[
a_{k}^{\mathbb{R}}\sim\Psi\left(n,k,p,q\right),
\]
where the constants of equivalence only depend on $p$ and $q$.

(ii) If $1\leq p\leq q\leq2$ or $2\leq p\leq q\leq\infty$, then

\[
a_{k}^{\mathbb{R}}\geq\sqrt{\left(1-\frac{k}{n}\right)}.
\]
\end{thm}

\begin{proof}
A reference for the proof is given in \cite[Subsect. 3.2.3, Thm. 1]{Edmunds and Triebel}.
\end{proof}
Since these were only the approximation numbers for real valued sequence
spaces, we want to give a relationship to the complex ones, as it
is done in \cite[Subsect. 3.2.3., Prop.]{Edmunds and Triebel}
\begin{lem}
\label{lem:relationship}Let $k\in\mathbb{N}$, $k\leq n$ and suppose
that $p,q\in\left[1,\infty\right]$. Then 

\[
a_{2k-1}^{\mathbb{R}}\leq a_{k}\leq2a_{2k}^{\mathbb{R}}
\]

(where $\alpha_{k}^{\mathbb{R}}=0$ if $k>n$).
\end{lem}

\begin{proof}
Again we shall use this statement without proof in this bachelor's
thesis and refer to the above mentioned proposition.
\end{proof}
For a matter of completeness we shall add the following two lemmata.
They are directly taken from \cite[Lem 3.3., Lem 3.4.]{JanVybiral}.
\begin{lem}
If $1\leq k\leq n<\infty$ and $0<q\leq p\leq\infty$, then

\[
a_{k}=\left(n-k\right)^{1/q-1/p}.
\]
\end{lem}

\begin{proof}
According to the author, this statement is a generalization of the
proof for $1\leq q\leq p\leq\infty$ from \cite[Subsect. 11.11.5., Lem.]{Albrecht Pietsch 2}.
It is also left without proof in this bachelor's thesis.
\end{proof}
\begin{lem}
Let $0<p\leq1$.

(i) Let $0<\lambda<1.$ Then there exists a number $c_{\lambda}>0$
such that for all $k,n\in\mathbb{N}$ with $n^{\lambda}<k\leq n$,
we have

\[
a_{k}\left(\mbox{id : }\ell_{p}^{n}\rightarrow\ell_{\infty}^{n}\right)\leq\frac{c_{\lambda}}{\sqrt{k}}.
\]

(ii) There is a number $c>0$ such that for $n\geq1$
\end{lem}

\begin{proof}
Again we only refer to \cite[Lem. 3.4.]{JanVybiral} for the proof.
\end{proof}
\begin{cor}
\label{cor:Corollar-ak}Let $n\in\mathbb{N}$, $1\leq p,q\leq\infty$
and $k\leq\frac{n}{4}$, then

\[
a_{k}\sim\begin{cases}
1 & \mbox{if }1\leq p<q\leq2\\
\min\left(1,n^{\frac{1}{q}}k^{-\frac{1}{2}}\right) & \mbox{if }1\leq p<2\leq q<p'\\
\min\left(1,n^{\frac{1}{p'}}k^{-\frac{1}{2}}\right) & \mbox{if }1\leq p\leq2\leq p'\leq q\leq\infty\,\mbox{and\,}\left(p,q\right)\neq\left(1,\infty\right)\\
1 & \mbox{if }2\leq p\leq q\leq\infty
\end{cases}
\]

where $p'$ is defined through the equation $\frac{1}{p}+\frac{1}{p'}=1$.
\end{cor}

\begin{proof}
The proof rests upon Theorem \ref{thm:real-valued-approximation-numbers}
and Lemma \ref{lem:relationship}.
\end{proof}
We proceed by extending these results for cases, where $p,q\in\left(0,1\right)$
in the following theorem. (See \cite[Subsect. 3.2.3, Thm. 2]{Edmunds and Triebel}.)
\begin{thm}
Let $n\in\mathbb{N},$ then\\
(i) ~~If $0<p\leq q\leq2$ and $k\leq\frac{n}{4}$, then $a_{k}\sim1$.\\
(ii) ~If $0<q\leq p\leq\infty$ and $n=2k$, then $a_{k}\geq2^{-\frac{1}{q}}n^{\frac{1}{q}-\frac{1}{p}}$.\\
(iii) If $0<p<2<q<p'$ and $k\leq\frac{n}{4},$then $a_{k}\sim\min\left(1,n^{\frac{1}{q}}k^{-\frac{1}{2}}\right)$.
\end{thm}

\begin{proof}
In the \emph{first step}, we prove (iii). Let $\beta_{k}:=a_{k}\left(\mbox{id : }\ell_{1}^{n}\rightarrow\ell_{q}^{n}\right)$
and $x^{r}\in\ell_{1}^{n}$ with $j$ - th component $\delta_{jr}$
for $j,r=1,\ldots,n$ (which denotes the Kronecker delta). It is plain
to see, that 

\[
\overline{B}_{\ell_{1}^{n}}=\left\{ x=\sum_{r=1}^{n}\lambda_{r}x^{r}:\,\sum_{r=1}^{n}\vert\lambda_{r}\vert=1\right\} .
\]

So let $T\,:\,\ell_{1}^{n}\rightarrow\ell_{q}^{n}$ be linear and
with $\mbox{rank }T<k$, then we get for $x\in\overline{B}_{\ell_{1}^{n}}$

\begin{eqnarray}
 & x-Tx= & \sum_{r=1}^{n}\lambda_{r}\underset{\tiny=:w^{r}}{\underbrace{\left(x^{r}-Tx^{r}\right)}}=\sum_{r=1}^{n}\lambda_{r}w^{r}.\nonumber \\
\Longrightarrow & \beta_{k}\leq & \sup_{x\in\overline{B}_{\ell_{1}^{n}}}\Vert x-Tx\Vert_{q}\leq\sup_{r=1,\ldots,n}\Vert w^{r}\Vert_{q}\nonumber \\
 & = & \sup_{r=1,\ldots,n}\Vert\left(\left(\mbox{id : }\ell_{p}^{n}\rightarrow\ell_{q}^{n}\right)-T\right)x^{r}\Vert_{q}\nonumber \\
 & \leq & \Vert\left(\left(\mbox{id : }\ell_{p}^{n}\rightarrow\ell_{q}^{n}\right)-T\right)\Vert.\label{eq:estimation1000}
\end{eqnarray}

Thus by taking the infimum over all legitimate operators $T$ and
with Corollary \ref{cor:Corollar-ak} we arrive at

\begin{equation}
\min\left(1,n^{\frac{1}{q}}k^{-\frac{1}{2}}\right)\sim\beta_{k}\leq a_{k}.\label{eq:equivalence}
\end{equation}

On the other hand, we obtain the opposite inequality by the fact that
in this case $p<1$ and the monotonicity of the $\ell_{p}^{n}$ spaces.
In particular we use $\ell_{p}^{n}\hookrightarrow\ell_{1}^{n}$. This
completes step one. 

We proceed with the \emph{second step}, in which we prove (i). At
first, we assume that $0<p<1<q\leq2$ from which we get \eqref{eq:equivalence}
with $1$ on the left-hand side with Corollary \ref{cor:Corollar-ak}.
Furthermore we get $a_{k}\leq\beta_{k}$, since $p<1$ and the monotonicity
of the $\ell_{p}^{n}$ spaces (As above $\ell_{p}^{n}\hookrightarrow\ell_{1}^{n}$).
Hence $a_{k}\sim1$. We now investigate the remaining case $0<p\leq q\leq1$
and set $\beta_{k}=a_{k}\left(\mbox{id : }\ell_{q}^{n}\rightarrow\ell_{q}^{n}\right)$.
As we have shown earlier in Theorem \ref{thm:Properties a} (N$_{a}$),
we know that $\beta_{k}=1$ for those $k$ admissible here. The analogue
to \eqref{eq:equivalence} is 

\begin{eqnarray*}
1 & \leq & \Vert\sum_{r=1}^{n}\lambda_{r}w^{r}\Vert_{q}^{q}\leq\sum_{r=1}^{n}\vert\lambda_{r}\vert^{q}\Vert w^{r}\Vert_{q}^{q}\\
 & \leq & \sup_{r=1,\ldots,n}\Vert w^{r}\Vert_{q}^{q}=\sup_{r=1,\ldots,n}\Vert x^{r}-Tx^{r}\Vert_{q}^{q}\\
 & \leq & \sup_{\Vert x\Vert_{p}\leq1}\Vert\left(\mbox{id : }\ell_{p}^{n}\rightarrow\ell_{q}^{n}\right)-T\Vert.
\end{eqnarray*}
 Hence, by taking the infimum over all legitimate operators $T$,
we get $a_{k}\left(\mbox{id : }\ell_{p}^{n}\rightarrow\ell_{q}^{n}\right)\geq1$.
We conclude the second step by using the monotonicity of the $\ell_{p}^{n}$
spaces again, which finally yields $a_{k}\sim1$.

The third step will be, to prove the last remaining part of this theorem,
which is (ii). Therefore let $T:\,\ell_{p}^{n}\rightarrow\ell_{q}^{n}$
be represented as an $n\mbox{ x }n$ matrix with $\mbox{rank }T<k=\frac{n}{2}$.
We know that $\dim\ker T>k=\frac{n}{2}$ and use V.D. Milman's lemma
(see \cite[Sect. 2.9.,  Lem. 2.9.6.]{Albrecht Pietsch}) . From there,
it follows that there exists an element $x=\left(x_{1},\ldots,x_{n}\right)\in\ker T$
with $\vert x_{j}\vert\leq1$ for all $j=1,\ldots,n$ and, (for example)
$\vert x_{1}\vert=\ldots=\vert x_{n/2}\vert=1$. For this $x$ we
know

\begin{equation}
\Vert x\Vert_{q}\geq\left(n/2\right)^{1/q}\,\,\,\mbox{and}\,\,\,\Vert x\Vert_{p}\leq n^{1/p}.\label{eq:pq norm}
\end{equation}

Therefore we conclude

\[
\Vert x\Vert_{q}=\Vert\left(I-T\right)x\Vert_{q}\leq\Vert\left(I-T\right)\Vert\Vert x\Vert_{p}.
\]

Here $I$ stands for $\left(\mbox{id : }\ell_{p}^{n}\rightarrow\ell_{q}^{n}\right)$,
but since we identified $T$ as a matrix, we do the same with $I$.
By taking the infimum over all such operators $T$ and the above estimate,
we finally get

\[
a_{k}\left(\mbox{id : }\ell_{p}^{n}\rightarrow\ell_{q}^{n}\right)\geq\frac{\Vert x\Vert_{q}}{\Vert x\Vert_{p}}\overset{\eqref{eq:pq norm}}{=}2^{-\frac{1}{q}}n^{\frac{1}{q}-\frac{1}{p}},
\]

which concludes the third step, as well as the theorem.
\end{proof}
As far as approximation numbers of the identity operator between finite
sequence spaces are concerned, there is one last corollary, that we
will add in this bachelor's thesis, which is a conclusion of Corollary
\ref{cor:Corollar-ak}. It is taken from \cite[Cor. 2.2.]{Caetano}.
\begin{cor}
Let $0<p\leq2\leq q<\infty$ (or $1<p\leq2<q=\infty$ ). Then

(i) there exists $c>0$ such that for all $k,n\in\mathbb{N}$

\[
a_{k}\leq cn^{1/\min\left\{ p',q\right\} }k^{-\frac{1}{2}}.
\]

(ii) there exists $c>0$ such that for all $k,n\in\mathbb{N}$ with
$k\leq\frac{1}{4}n^{2/\min\left\{ p',q\right\} }$

\[
a_{k}\geq c.
\]
\end{cor}

\begin{proof}
To prove (i) we will first have a look at $k>n$. As we have shown
in Theorem \ref{thm:Properties a} we then have $a_{k}=0$ and therefore
the required estimate is valid. We now consider the remaining case,
which is $k\leq n$. Our aim is to use Corollary \ref{cor:Corollar-ak}
and therefore we have a look at the composition 

\[
\ell_{p}^{n}\overset{J}{\longrightarrow}\ell_{p}^{4n}\overset{\mbox{id}}{\longrightarrow}\ell_{q}^{4n}\overset{P}{\longrightarrow}\ell_{q}^{n}.
\]

We define $J\left(\xi\right)=\left(\xi_{1},\ldots,\xi_{n},\underset{3n\mbox{ times}}{\underbrace{0,\ldots,0}}\right)$
for $\xi=\left(\xi_{1},\ldots,\xi_{n}\right)\in\ell_{p}^{n}$ as well
as $P\left(\xi\right)=\left(\xi_{1},\ldots,\xi_{n}\right)$ for $\xi\in\ell_{q}^{4n}$.
This yields

\begin{eqnarray*}
a_{k}\left(\mbox{id : }\ell_{p}^{n}\rightarrow\ell_{q}^{n}\right) & = & a_{k+1-1}\left(P\cdot\left(\mbox{id : }\ell_{p}^{4n}\rightarrow\ell_{q}^{4n}\right)\cdot J\right)\\
 & \leq & a_{1}\left(P\right)\cdot a_{k+1-1}\left(\left(\mbox{id : }\ell_{p}^{4n}\rightarrow\ell_{q}^{4n}\right)\cdot J\right)\\
 & \leq & \Vert P\Vert\Vert J\Vert a_{k}\left(\mbox{id : }\ell_{p}^{4n}\rightarrow\ell_{q}^{4n}\right),
\end{eqnarray*}

where we used the properties (M$_{a}$) and (P$_{a}$) of Theorem
\ref{thm:Properties a}. Obviously $P$ and $J$ are bounded operators
and we are now in position to use Corollary \ref{cor:Corollar-ak}.
This amounts to

\[
a_{k}\left(\mbox{id : }\ell_{p}^{n}\rightarrow\ell_{q}^{n}\right)\leq c\cdot n^{1/\min\left\{ p',q\right\} }k^{-\frac{1}{2}},
\]

which is (i). The next step will be proving (ii). Our premise is still
$0<p\leq2\leq q<\infty$ (or completely analogue $1<p\leq q=\infty$).
Therefore we have $\min\left\{ p',q\right\} \geq2$, since the conjugate
index $p'$ is set $p'=\infty$ if $0<p\leq1$. However $k\leq\frac{1}{4}n^{2/\min\left\{ p',q\right\} }$,
amounts to $k\leq\frac{n}{4}$, which brings us again in position
to use Corollary \ref{cor:Corollar-ak}. Furthermore we get 

\[
n^{1/\min\left\{ p',q\right\} }k^{-\frac{1}{2}}\geq n^{1/\min\left\{ p',q\right\} }\cdot\left(\frac{1}{2}n^{2/\min\left\{ p',q\right\} }\right)^{-\frac{1}{2}}=\sqrt{2}>1,
\]

which finally yields $a_{k}\geq c$ for some $c>0$ independent of
$k$ and $n$.

\end{proof}

\section{$\mbox{id : }\ell_{p}^{n}\rightarrow\ell_{q}^{n}$ and Kolmogorov
numbers }

The last remaining numbers in this bachelor's thesis, concerning $\mbox{id : }\ell_{p}^{n}\rightarrow\ell_{q}^{n}$
are Kolmogorov numbers. With our investigation we will mostly follow
the notations of \cite[Sect. 4.]{JanVybiral} and start with a lemma
taken from there. (\cite[Lem. 4.2.]{JanVybiral})
\begin{lem}
Let $1\leq k\leq n<\infty$ and $1\leq p,q\leq\infty$. We define 

\[
\Phi\left(n,k,p,q\right):=\begin{cases}
\left(n-k+1\right)^{\frac{1}{q}-\frac{1}{p}} & \mbox{if }1\leq q\leq p\leq\infty\\
\left(\min\left\{ 1;n^{\frac{1}{q}}k^{-\frac{1}{2}}\right\} \right)^{\frac{\frac{1}{p}-\frac{1}{q}}{\frac{1}{2}-\frac{1}{q}}} & \mbox{if }2\leq p<q\leq\infty\\
\max\left\{ n^{\frac{1}{q}-\frac{1}{p}};\sqrt{1-\frac{k}{n}}^{\frac{\frac{1}{p}-\frac{1}{q}}{\frac{1}{p}-\frac{1}{2}}}\right\}  & \mbox{if }1\leq p<q\leq2\\
\max\left\{ n^{\frac{1}{q}-\frac{1}{p}},\min\left\{ 1,n^{\frac{1}{q}}k^{-\frac{1}{2}}\right\} \cdot\sqrt{1-\frac{k}{n}}\right\}  & \mbox{if }1\leq p<2<q\leq\infty.
\end{cases}
\]

Then $d_{k}\left(\mbox{id : }\ell_{p}^{n}\rightarrow\ell_{q}^{n}\right)\sim\Phi\left(n,k,p,q\right)$
if $q<\infty$, where the constants of equivalence are independent
of $k$ and $n$ but may depend on $p$ and $q$. 

Furthermore there exist $c_{p},C_{p}>0$ such that 

\[
c_{p}\Phi\left(n,k,p,\infty\right)\leq d_{k}\left(\mbox{id : }\ell_{p}^{n}\rightarrow\ell_{\infty}^{n}\right)\leq C_{p}\Phi\left(n,k,p,\infty\right)\left(\log\left(\frac{en}{k}\right)\right)^{3/2}
\]

for $1\leq p\leq\infty$. 
\end{lem}

\begin{proof}
As it is done in \cite[Lem. 4.2.]{JanVybiral}, we refer to \cite[Thm. 1]{Gluskin}.\\
\\
\end{proof}
The above Lemma only contained cases, in which $1\leq p,q\leq\infty$.
We shall now add some estimates which apply to quasi-Banach spaces.
Again, the following Lemma is taken from \cite[Lem. 4.3.]{JanVybiral}.\\

\begin{lem}
If $0<q\leq p\leq\infty$, then there exists a constant $c>0$ such
that 

\[
d_{\left\lceil cn\right\rceil +1}\left(\mbox{id : }\ell_{p}^{2n}\rightarrow\ell_{q}^{2n}\right)\gtrsim n^{\frac{1}{q}-\frac{1}{p}},\,\,\,\mbox{for }k\in\mathbb{N}
\]

where $\left\lceil cn\right\rceil $ denotes the upper integer part
of $cn$.
\end{lem}

\begin{proof}
First, we consider the case in which $q\geq1$. Then we have a special
case of \cite[Subsect 11.11.4., Lem. 1.]{Albrecht Pietsch 2} , which
states

\[
d_{n}\left(\mbox{id : }\ell_{p}^{m}\rightarrow\ell_{q}^{m}\right)\gtrsim\left(m-n+1\right)^{\frac{1}{q}-\frac{1}{p}}\,\,\,\mbox{for }1\leq n\leq m.
\]

Since this argument does not stand for $q<1$ we recall two facts.
The first is Proposition \ref{pro:Triebel Proposition}, where we
have shown, that 

\[
e_{k}\left(\mbox{id : }\ell_{p}^{2n}\rightarrow\ell_{q}^{2n}\right)\geq c_{1}\cdot2^{-\frac{k}{4n}}\left(4n\right){}^{\frac{1}{p}-\frac{1}{q}}
\]

for some constant $c_{1}>0$ depending only of $q$ and $p$. The
second fact is Proposition \ref{pro:supremumsdinge2} with which we
derive that

\[
c_{2}\cdot n^{\alpha}n^{\frac{1}{q}-\frac{1}{p}}\leq\sup_{1\leq k\leq n}k^{\alpha}d_{k}\left(\mbox{id : }\ell_{p}^{2n}\rightarrow\ell_{q}^{2n}\right).
\]

That means, that for every $n\in\mathbb{N}$ there exists $k_{n}\leq n$
such that

\[
c_{2}\cdot n^{\alpha}n^{\frac{1}{q}-\frac{1}{p}}\leq k_{n}^{\alpha}d_{k_{n}}\left(\mbox{id : }\ell_{p}^{2n}\rightarrow\ell_{q}^{2n}\right).
\]

Furthermore there exists a constant $c\in(0,1]$, such that for all
$n\in\mathbb{N}$ $n\geq k\geq cn$. Combining this conclusion with
the preceding estimate finally amounts to

\[
c_{2}\cdot n^{\frac{1}{q}-\frac{1}{p}}\leq d_{\left[cn\right]+1}\left(\mbox{id : }\ell_{p}^{2n}\rightarrow\ell_{q}^{2n}\right).
\]
\\
\\
\end{proof}

\part{Relationship to Spectral Theory}

\section{Preliminary Considerations}

This last part will give a connection between entropy numbers as well
as approximation numbers and eigenvalues of compact operators of infinitely
dimensional Hilbert spaces by the inequalities of Carl and Weyl.  But
when it comes to the relationship between the here considered quantities,
we cannot avoid certain definitions.\\

\begin{defn}
Let $\mathbb{X}$ be an arbitrary, complex quasi-Banach space and
$T\in\mathcal{L}\left(\mathbb{X}\right)$. We then call

(i) $\varrho\left(T\right)=\left\{ \lambda\in\mathbb{C}\,:\,\exists\left(T-\lambda\mbox{id}_{\mathbb{X}}\right)^{-1}\in\mathcal{L}\left(\mathbb{X}\right)\right\} $
the resolvent set of $T$.

(ii) $\sigma\left(T\right)=\mathbb{C}\backslash\varrho\left(T\right)$
the spectrum of $T$.

(iii) $\lambda$ an eigenvalue of $T$, if there exists an element
$x\in\mathbb{X}$, $x\neq0$ with $Tx=\lambda x$.

(iv) $r(T)=\lim_{n\rightarrow\infty}\sqrt[n]{\Vert T^{n}\Vert}$ the
spectral radius of $T$.
\end{defn}

A result, following from this definition, is the next proposition.
To maintain the intended scope of this bachelor's thesis we shall
only give the statement without proving it. A proof however can be
found in \cite[Sect. 1.2., Thm.]{Edmunds and Triebel}.
\begin{prop}
\label{pro:Eigenvalues}Let $\mathbb{X}$ be a complex infinite-dimensional
quasi-Banach space and $T\in\mathcal{K}\left(\mathbb{X}\right)$.
Then the spectrum $\sigma\left(T\right)$ consists only of $\left\{ 0\right\} $
and at most countably infinite number of eigenvalues of finite algebraic
multiplicity, which accumulate only at $0$. That is

\[
\sigma\left(T\right)=\left\{ 0\right\} \cup\left\{ \left(\lambda_{k}\right)_{k\in\mathbb{N}}\subset\mathbb{C}\,:\,\lambda_{k}\neq0,\,\lambda_{k}\mbox{ eigenvalue of }T,\,\dim N\left(T-\lambda_{k}\mbox{id}\right)<\infty\right\} .
\]
\end{prop}

\begin{proof}
Without proof. (See reference above.)
\end{proof}
Because of this proposition, we can construct a sequence out of all
non-zero eigenvalues $\left(\lambda_{k}\right)_{k\in\mathbb{N}}$,
such that 

\begin{equation}
\left|\lambda_{1}\left(T\right)\right|\geq\left|\lambda_{2}\left(T\right)\right|\geq\ldots\geq0,\label{eq:sequence}
\end{equation}

where we repeated and ordered the $\lambda_{k}$ according to their
algebraic multiplicity. If $T$ has only $m\left(<\infty\right)$
distinct eigenvalues and $M$ is the sum of their algebraic multiplicities,
then we simply put $\lambda_{n}\left(T\right)=0$ for every $n>m$.
From now on, we will refer to this ordered sequence as the \emph{eigenvalue
sequence of $T$.}

\section{The Inequality of Carl}

The perhaps most useful connection to the here concerned quantities
is the following inequality, which was first stated in \cite{Carl=000026Triebel}.
Later on it was extended to fit in the context of quasi-Banach spaces
by \cite{Edmunds and Triebel}.
\begin{thm}
\label{thm:Carl Inequality}(Inequality of Carl)

Let $\mathbb{X}$ be an arbitrary complex quasi-Banach space and $T\in\mathcal{K}\left(\mathbb{X}\right)$
with its eigenvalue sequence $\lambda_{1}\left(T\right),\lambda_{2}\left(T\right),\ldots,\lambda_{n}\left(T\right),\ldots$.
Then

\[
\left(\prod_{m=1}^{k}\vert\lambda_{m}\left(T\right)\vert\right)^{\frac{1}{k}}\leq\inf_{n\in\mathbb{N}}2^{\frac{n}{2k}}e_{n}\left(T\right)\,\mbox{for }\,k\in\mathbb{N}.
\]
\end{thm}

\begin{proof}
As we have mentioned above, the proof can be found in \cite[Subsect. 1.3.4., Thm.]{Edmunds and Triebel}.
\end{proof}
From this useful theorem, we can draw an immediate conclusion, which
is the following corollary.
\begin{cor}
Let $T$ be as above. For all $k\in\mathbb{N}$ we have

\[
\left|\lambda_{k}\left(T\right)\right|\leq\sqrt{2}e_{k}\left(T\right)
\]
\end{cor}

\begin{proof}
Let $k\in\mathbb{N}$. With \eqref{eq:sequence} we get

\[
\vert\lambda_{k}\left(T\right)\vert=\vert\lambda_{k}\left(T\right)\vert^{k\cdot\frac{1}{k}}=\left(\prod_{i=1}^{k}\vert\lambda_{k}\left(T\right)\vert\right)^{1/k}\leq\left(\prod_{m=1}^{k}\vert\lambda_{m}\left(T\right)\vert\right)^{1/k}.
\]

The estimate follows immediately, if we set $k=n$ in Theorem \ref{thm:Carl Inequality}.
\end{proof}
\begin{rem}
The above inequality was first discovered by \cite[Thm. 4]{Carl}
in the context of Banach spaces. Using this result it can be shown
that

\[
\lim_{n\rightarrow\infty}\left(e_{k}\left(T^{n}\right)\right)^{1/n}=r(T)\,\,\,\mbox{for }k\in\mathbb{N}.
\]

(See \cite[Subsect. 1.3.4., Rem. 1]{Edmunds and Triebel}.)
\end{rem}

Let us now examine the relationship between eigenvalues of an operator
$T\in\mathcal{K}\left(\mathbb{X}\right)$, where $\mathbb{X}$ denotes
an arbitrary complex Banach space and the approximation numbers. It
has been shown by \cite[Prop. 2.d.6., p. 134]{K=0000F6nig} that for
$n\in\mathbb{N}$

\[
\vert\lambda_{n}\left(T\right)\vert=\lim_{k\rightarrow\infty}a_{n}^{1/k}\left(T^{k}\right),
\]

as well as for $p\in\left(0,\infty\right)$, that for some constant
$K_{p}$

\[
\left(\sum_{k=1}^{N}\vert\lambda_{k}\left(T\right)\vert^{p}\right)^{1/p}\leq K_{p}\left(\sum_{k=1}^{N}\left[a_{k}\left(T\right)\right]^{p}\right)^{1/p}.
\]

\section{Hilbert Space Setting and the Inequality of Weyl }

At last, let us study the Hilbert space setting. It is only natural
that better results arise in this case. Since we are now dealing with
Hilbert spaces, we can talk about inner products and therefore about
adjoint operators. (For further information about the adjoint operator,
see \cite[Subsect. 2.2.3.]{H=0000F6here Analysis}). 
\begin{defn}
Let $\mathbb{H}$ be a Hilbert space and $T\in\mathcal{L}\left(\mathbb{H}\right)$.
We will denote the adjoint operator of $T$ with $T^{*}$ and call
$T$ self-adjoint if $T=T^{*}.$
\end{defn}

\begin{prop}
Let $\mathbb{H}$ be a Hilbert space and $T\in\mathcal{K}\left(\mathbb{H}\right)$be
self-adjoint. For given $n\in\mathbb{N}$, we have

\[
\vert\lambda_{n}\left(T\right)\vert=a_{n}\left(T\right),
\]

where $\lambda_{n}\left(T\right)$ is the $n$ - th eigenvalue of
the eigenvalue-sequence of $T$.
\end{prop}

\begin{proof}
For the proof, see \cite[Sect. 4.4., Prop. 4.4.1.]{Carl=000026Stephani}.
\end{proof}
\begin{rem}
If we additionally demand that $T$ is a non-negative operator (that
is to say, that $\left\langle Tx,x\right\rangle \geq0$ for all $x\in\mathbb{H}$)
we even get the equality $\lambda_{n}\left(T\right)=a_{n}\left(T\right)$.
(See \cite[Subsect. 1.3.4., Rem. 1.]{Edmunds and Triebel})
\end{rem}

The perhaps most popular inequality concerning a relationship between
approximation numbers and eigenvalues of an operator is the inequality
of Weyl. 
\begin{thm}
(Inequality of Weyl)

Let $\mathbb{H}$ be a Hilbert space and $T\in\mathcal{K}\left(\mathbb{H}\right)$
with its eigenvalue sequence $\left(\lambda_{k}\left(T\right)\right)_{k\in\mathbb{N}}$.
Then for given $n\in\mathbb{N}$ we have

\[
\prod_{i=1}^{n}\vert\lambda_{i}\left(T\right)\vert\leq\prod_{i=1}^{n}a_{i}\left(T\right)
\]

and even equality if $\dim\mathbb{H}=n<\infty$. Furthermore, for
given $p\in\left(0,\infty\right)$, we get

\[
\sum_{i=1}^{n}\vert\lambda_{i}\left(T\right)\vert^{p}\leq\sum_{i=1}^{n}\left[a_{i}\left(T\right)\right]^{p}.
\]
\end{thm}

\begin{proof}
For the proof see \cite[Sect. 4.4., Prop. 4.4.2.]{Carl=000026Stephani}.
\end{proof}
\pagebreak{}


\begin{thebibliography}{References}
\bibitem[BBP95]{BasBernAna} J. Bastero, J. Bernu\'{e}s and A. Pe\~{n}a, \emph{An
extension of Milman's reverse Brunn-Minkowski inequality}, Geometrical
Functional Analysis, 5(3):572-581, 1995

\bibitem[Car81]{Carl} B. Carl, \emph{Entropy Numbers, s-Numbers,
and Eigenvalue Problems}, Journal of Functional Analysis, 41:290-306,
1981 

\bibitem[CS90]{Carl=000026Stephani}B. Carl and I. Stephani, \emph{Entropy,
compactness and the approximation of operators}, Cambridge University
Press , Cambridge, 1990

\bibitem[CT80]{Carl=000026Triebel}B. Carl and H. Triebel, \emph{Inequalities
between eigenvalues, entropy numbers, and related quantities of compact
operators in Banach spaces}, Math. Ann. 251:129-133, 1980

\bibitem[Cae98]{Caetano} A. M. Caetano, \emph{About Approximation
Numbers in Function Spaces}, Journal of Approximation Theory 94:383-395,
1998

\bibitem[DL93]{Constructive Approximation}R. DeVore and G.G. Lorentz,
\emph{Constructive approximation}, volume 303 of \emph{Grundlehren
der Mathematischen Wissenschaften}. Springer, Berlin, 1993

\bibitem[ET96]{Edmunds and Triebel}D.E. Edmunds and H. Triebel, \emph{Function
spaces, entropy numbers, differential operators}, Cambridge University
Press, Cambridge, 1996

\bibitem[Glu84]{Gluskin}E. D. Gluskin, \emph{Norms of random matrices
and widths of finite-dimensional sets}, Math USSR Sb. 48:173-182,
1984

\bibitem[Har10]{HaroskeAT}D. D. Haroske, \emph{Approximationstheorie
Vorlesungsskript, }Friedrich-Schiller-Universität Jena, winter term
2009/2010

\bibitem[Har11]{HaroskeHA}D. D. Haroske, \emph{H\"{o}here Analysis Vorlesungsskript},
Friedrich-Schiller-Universit\"{a}t Jena, academic year 2010/2011

\bibitem[K\"{o}n86]{K=0000F6nig}H. K\"{o}nig, \emph{Eigenvalue distribution
of compact operators}, Birkhäuser Verlag, Basel, 1986 

\bibitem[K\"{u}h01]{JournalKuehn} T. Kühn, \emph{A lower estimate for
entropy numbers}, Journal of Approximation Theory 110:120-124, 2001

\bibitem[Pie78]{Albrecht Pietsch 2} A. Pietsch, \emph{Operator ideals},
volume 16 of Mathematical Monographs, Deutscher Verlag der Wissenschaften,
Berlin, 1978

\bibitem[Pie87]{Albrecht Pietsch} A. Pietsch, \emph{Eigenvalues and
s-Numbers}, Akademische Verlagsgesellschaft Geest \& Portig, Leipzig,
1987 

\bibitem[Pis89]{Pisier}G. Pisier, \emph{The volume of convex bodies
and Banach space geometry}, Cambridge University Press, Cambridge,
1989

\bibitem[Tri92]{H=0000F6here Analysis} H. Triebel, \emph{Higher Analysis},
Johann Ambrosius Barth Verlag, Bad Langensalza, 1992

\bibitem[Tri97]{Triebel1997} H. Triebel, \emph{Fractals and Spectra},
Birkhäuser Verlag, Basel, 1997

\bibitem[Vyb08]{JanVybiral} J. Vyb\'{i}ral, \emph{Widths of embedding
in function spaces}, Journal of Complexity, 24:545-570, 2008\\
~\newpage{}
\end{thebibliography}
\end{document}